\documentclass[oneside,english,a4paper]{amsart}
\usepackage[T1]{fontenc}
\usepackage[latin9]{inputenc}
\usepackage{verbatim}
\usepackage{mathrsfs}
\usepackage{amstext}
\usepackage{amsthm}
\usepackage{amssymb}

\makeatletter
\numberwithin{equation}{section}
\numberwithin{figure}{section}
\numberwithin{table}{section}
\theoremstyle{plain}
\newtheorem{thm}{\protect\theoremname}[section]
\theoremstyle{definition}
\newtheorem{defn}[thm]{\protect\definitionname}
\theoremstyle{remark}
\newtheorem{rem}[thm]{\protect\remarkname}
\theoremstyle{plain}
\newtheorem{lem}[thm]{\protect\lemmaname}
\theoremstyle{plain}
\newtheorem{prop}[thm]{\protect\propositionname}
\theoremstyle{plain}
\newtheorem{cor}[thm]{\protect\corollaryname}
\theoremstyle{definition}
\newtheorem{example}[thm]{\protect\examplename}
\theoremstyle{plain}
\newtheorem{conjecture}[thm]{\protect\conjecturename}

\usepackage{extarrow}
\usepackage{enumerate}
\usepackage{fontenc}
\usepackage[T1]{fontenc}
\usepackage{graphicx}
\usepackage[colorlinks=true, linkcolor=blue]{hyperref}
\usepackage{latexsym}
\usepackage{mathrsfs}
\usepackage{mathtext}
\usepackage{tikz}
\usepackage{textcomp}
\usepackage{upgreek}
\usepackage{xy}

\usetikzlibrary{arrows}
\usetikzlibrary{matrix}
\usetikzlibrary{shapes}
\usetikzlibrary{snakes}
\usetikzlibrary{matrix}

\DeclareMathOperator{\Add}{\textup{Add}}
\DeclareMathOperator{\Aff}{\textup{Aff}}
\DeclareMathOperator{\Alg}{\textup{Alg}}
\DeclareMathOperator{\Ann}{\textup{Ann}}
\DeclareMathOperator{\Arr}{\textup{Arr}}
\DeclareMathOperator{\Art}{\textup{Art}}
\DeclareMathOperator{\Ass}{\textup{Ass}}
\DeclareMathOperator{\Aut}{\textup{Aut}}
\DeclareMathOperator{\Autsh}{\underline{\textup{Aut}}}
\DeclareMathOperator{\Bi}{\textup{B}}
\DeclareMathOperator{\CAdd}{\textup{CAdd}}
\DeclareMathOperator{\CAlg}{\textup{CAlg}}
\DeclareMathOperator{\CMon}{\textup{CMon}}
\DeclareMathOperator{\CPMon}{\textup{CPMon}}
\DeclareMathOperator{\CRings}{\textup{CRings}}
\DeclareMathOperator{\CSMon}{\textup{CSMon}}
\DeclareMathOperator{\CaCl}{\textup{CaCl}}
\DeclareMathOperator{\Cart}{\textup{Cart}}
\DeclareMathOperator{\Cl}{\textup{Cl}}
\DeclareMathOperator{\Coh}{\textup{Coh}}
\DeclareMathOperator{\Coker}{\textup{Coker}}
\DeclareMathOperator{\Cov}{\textup{Cov}}
\DeclareMathOperator{\Der}{\textup{Der}}
\DeclareMathOperator{\Div}{\textup{Div}}
\DeclareMathOperator{\End}{\textup{End}}
\DeclareMathOperator{\Endsh}{\underline{\textup{End}}}
\DeclareMathOperator{\Ext}{\textup{Ext}}
\DeclareMathOperator{\Extsh}{\underline{\textup{Ext}}}
\DeclareMathOperator{\FAdd}{\textup{FAdd}}
\DeclareMathOperator{\FCoh}{\textup{FCoh}}
\DeclareMathOperator{\FGrad}{\textup{FGrad}}
\DeclareMathOperator{\FLoc}{\textup{FLoc}}
\DeclareMathOperator{\FMod}{\textup{FMod}}
\DeclareMathOperator{\FPMon}{\textup{FPMon}}
\DeclareMathOperator{\FRep}{\textup{FRep}}
\DeclareMathOperator{\FSMon}{\textup{FSMon}}
\DeclareMathOperator{\FVect}{\textup{FVect}}
\DeclareMathOperator{\Fibr}{\textup{Fibr}}
\DeclareMathOperator{\Fix}{\textup{Fix}}
\DeclareMathOperator{\Fl}{\textup{Fl}}
\DeclareMathOperator{\Fr}{\textup{Fr}}
\DeclareMathOperator{\Funct}{\textup{Funct}}
\DeclareMathOperator{\GAlg}{\textup{GAlg}}
\DeclareMathOperator{\GExt}{\textup{GExt}}
\DeclareMathOperator{\GHom}{\textup{GHom}}
\DeclareMathOperator{\GL}{\textup{GL}}
\DeclareMathOperator{\GMod}{\textup{GMod}}
\DeclareMathOperator{\GRis}{\textup{GRis}}
\DeclareMathOperator{\GRiv}{\textup{GRiv}}
\DeclareMathOperator{\Gal}{\textup{Gal}}
\DeclareMathOperator{\Gl}{\textup{Gl}}
\DeclareMathOperator{\Grad}{\textup{Grad}}
\DeclareMathOperator{\Hilb}{\textup{Hilb}}
\DeclareMathOperator{\Hl}{\textup{H}}
\DeclareMathOperator{\Hom}{\textup{Hom}}
\DeclareMathOperator{\Homsh}{\underline{\textup{Hom}}}
\DeclareMathOperator{\ISym}{\textup{Sym}^*}
\DeclareMathOperator{\Imm}{\textup{Im}}
\DeclareMathOperator{\Irr}{\textup{Irr}}
\DeclareMathOperator{\Iso}{\textup{Iso}}
\DeclareMathOperator{\Isosh}{\underline{\textup{Iso}}}
\DeclareMathOperator{\Ker}{\textup{Ker}}
\DeclareMathOperator{\LAdd}{\textup{LAdd}}
\DeclareMathOperator{\LAlg}{\textup{LAlg}}
\DeclareMathOperator{\LMon}{\textup{LMon}}
\DeclareMathOperator{\LPMon}{\textup{LPMon}}
\DeclareMathOperator{\LRings}{\textup{LRings}}
\DeclareMathOperator{\LSMon}{\textup{LSMon}}
\DeclareMathOperator{\Left}{\textup{L}}
\DeclareMathOperator{\Lex}{\textup{Lex}}
\DeclareMathOperator{\Loc}{\textup{Loc}}
\DeclareMathOperator{\M}{\textup{M}}
\DeclareMathOperator{\ML}{\textup{ML}}
\DeclareMathOperator{\MLex}{\textup{MLex}}
\DeclareMathOperator{\Map}{\textup{Map}}
\DeclareMathOperator{\Mod}{\textup{Mod}}
\DeclareMathOperator{\Mon}{\textup{Mon}}
\DeclareMathOperator{\Ob}{\textup{Ob}}
\DeclareMathOperator{\Obj}{\textup{Obj}}
\DeclareMathOperator{\PDiv}{\textup{PDiv}}
\DeclareMathOperator{\PGL}{\textup{PGL}}
\DeclareMathOperator{\PML}{\textup{PML}}
\DeclareMathOperator{\PMLex}{\textup{PMLex}}
\DeclareMathOperator{\PMon}{\textup{PMon}}
\DeclareMathOperator{\Pic}{\textup{Pic}}
\DeclareMathOperator{\Picsh}{\underline{\textup{Pic}}}
\DeclareMathOperator{\Pro}{\textup{Pro}}
\DeclareMathOperator{\Proj}{\textup{Proj}}
\DeclareMathOperator{\QAdd}{\textup{QAdd}}
\DeclareMathOperator{\QAlg}{\textup{QAlg}}
\DeclareMathOperator{\QCoh}{\textup{QCoh}}
\DeclareMathOperator{\QMon}{\textup{QMon}}
\DeclareMathOperator{\QPMon}{\textup{QPMon}}
\DeclareMathOperator{\QRings}{\textup{QRings}}
\DeclareMathOperator{\QSMon}{\textup{QSMon}}
\DeclareMathOperator{\R}{\textup{R}}
\DeclareMathOperator{\Rep}{\textup{Rep}}
\DeclareMathOperator{\Rings}{\textup{Rings}}
\DeclareMathOperator{\Riv}{\textup{Riv}}
\DeclareMathOperator{\SFibr}{\textup{SFibr}}
\DeclareMathOperator{\SMLex}{\textup{SMLex}}
\DeclareMathOperator{\SMex}{\textup{SMex}}
\DeclareMathOperator{\SMon}{\textup{SMon}}
\DeclareMathOperator{\SchI}{\textup{SchI}}
\DeclareMathOperator{\Sh}{\textup{Sh}}
\DeclareMathOperator{\Soc}{\textup{Soc}}
\DeclareMathOperator{\Spec}{\textup{Spec}}
\DeclareMathOperator{\Specsh}{\underline{\textup{Spec}}}
\DeclareMathOperator{\Stab}{\textup{Stab}}
\DeclareMathOperator{\Supp}{\textup{Supp}}
\DeclareMathOperator{\Sym}{\textup{Sym}}
\DeclareMathOperator{\TMod}{\textup{TMod}}
\DeclareMathOperator{\Top}{\textup{Top}}
\DeclareMathOperator{\Tor}{\textup{Tor}}
\DeclareMathOperator{\Vect}{\textup{Vect}}
\DeclareMathOperator{\alt}{\textup{ht}}
\DeclareMathOperator{\car}{\textup{char}}
\DeclareMathOperator{\codim}{\textup{codim}}
\DeclareMathOperator{\degtr}{\textup{degtr}}
\DeclareMathOperator{\depth}{\textup{depth}}
\DeclareMathOperator{\divis}{\textup{div}}
\DeclareMathOperator{\et}{\textup{et}}
\DeclareMathOperator{\ffpSch}{\textup{ffpSch}}
\DeclareMathOperator{\h}{\textup{h}}
\DeclareMathOperator{\ilim}{\displaystyle{\lim_{\longrightarrow}}}
\DeclareMathOperator{\ind}{\textup{ind}}
\DeclareMathOperator{\indim}{\textup{inj dim}}
\DeclareMathOperator{\lf}{\textup{LF}}
\DeclareMathOperator{\op}{\textup{op}}
\DeclareMathOperator{\ord}{\textup{ord}}
\DeclareMathOperator{\pd}{\textup{pd}}
\DeclareMathOperator{\plim}{\displaystyle{\lim_{\longleftarrow}}}
\DeclareMathOperator{\pr}{\textup{pr}}
\DeclareMathOperator{\pt}{\textup{pt}}
\DeclareMathOperator{\rk}{\textup{rk}}
\DeclareMathOperator{\tr}{\textup{tr}}
\DeclareMathOperator{\type}{\textup{r}}
\DeclareMathOperator*{\colim}{\textup{colim}}
\usepackage[colorinlistoftodos]{todonotes}
\usepackage{enumitem}
\setenumerate[1]{label={\arabic*)}} 
\usepackage[a4paper]{geometry}

\theoremstyle{plain}
\newtheorem{thma}{Theorem}
\newtheorem{thmb}{Theorem}

\makeatother

\usepackage{babel}
\providecommand{\conjecturename}{Conjecture}
\providecommand{\corollaryname}{Corollary}
\providecommand{\definitionname}{Definition}
\providecommand{\examplename}{Example}
\providecommand{\lemmaname}{Lemma}
\providecommand{\propositionname}{Proposition}
\providecommand{\remarkname}{Remark}
\providecommand{\theoremname}{Theorem}

\begin{document}
\title{Stacks of fiber functors and Tannaka's reconstruction}
\author{Fabio Tonini}
\address{Universitá degli Studi di Firenze, Dipartimento di Matematica e Informatica
\textquoteright Ulisse Dini\textquoteright , Viale Giovanni Battista
Morgagni, 67/A, 50134 Firenze, Italy}
\email{fabio.tonini@unifi.it}

\maketitle
\global\long\def\A{\mathbb{A}}%

\global\long\def\Ab{(\textup{Ab})}%

\global\long\def\C{\mathbb{C}}%

\global\long\def\Cat{(\textup{cat})}%

\global\long\def\Di#1{\textup{D}(#1)}%

\global\long\def\E{\mathcal{E}}%

\global\long\def\F{\mathbb{F}}%

\global\long\def\GCov{G\textup{-Cov}}%

\global\long\def\Gcat{(\textup{Galois cat})}%

\global\long\def\Gfsets#1{#1\textup{-fsets}}%

\global\long\def\Gm{\mathbb{G}_{m}}%

\global\long\def\GrCov#1{\textup{D}(#1)\textup{-Cov}}%

\global\long\def\Grp{(\textup{Grps})}%

\global\long\def\Gsets#1{(#1\textup{-sets})}%

\global\long\def\HCov{H\textup{-Cov}}%

\global\long\def\MCov{\textup{D}(M)\textup{-Cov}}%

\global\long\def\MHilb{M\textup{-Hilb}}%

\global\long\def\N{\mathbb{N}}%

\global\long\def\PGor{\textup{PGor}}%

\global\long\def\Fib{\textup{Fib}}%

\global\long\def\PGrp{(\textup{Profinite Grp})}%

\global\long\def\PP{\mathbb{P}}%

\global\long\def\Pj{\mathbb{P}}%

\global\long\def\Q{\mathbb{Q}}%

\global\long\def\RCov#1{#1\textup{-Cov}}%

\global\long\def\RR{\mathbb{R}}%

\global\long\def\Sch{\textup{Sch}}%

\global\long\def\WW{\textup{W}}%

\global\long\def\Z{\mathbb{Z}}%

\global\long\def\acts{\curvearrowright}%

\global\long\def\alA{\mathscr{A}}%

\global\long\def\alB{\mathscr{B}}%

\global\long\def\arr{\longrightarrow}%

\global\long\def\arrdi#1{\xlongrightarrow{#1}}%

\global\long\def\catC{\mathscr{C}}%

\global\long\def\catD{\mathscr{D}}%

\global\long\def\catF{\mathscr{F}}%

\global\long\def\catG{\mathscr{G}}%

\global\long\def\comma{,\ }%

\global\long\def\covU{\mathcal{U}}%

\global\long\def\covV{\mathcal{V}}%

\global\long\def\covW{\mathcal{W}}%

\global\long\def\duale#1{{#1}^{\vee}}%

\global\long\def\fasc#1{\widetilde{#1}}%

\global\long\def\fsets{(\textup{f-sets})}%

\global\long\def\iL{r\mathscr{L}}%

\global\long\def\id{\textup{id}}%

\global\long\def\la{\langle}%

\global\long\def\odi#1{\mathcal{O}_{#1}}%

\global\long\def\ra{\rangle}%

\global\long\def\set{(\textup{Sets})}%

\global\long\def\sets{(\textup{Sets})}%

\global\long\def\shA{\mathcal{A}}%

\global\long\def\shB{\mathcal{B}}%

\global\long\def\shC{\mathcal{C}}%

\global\long\def\shD{\mathcal{D}}%

\global\long\def\shE{\mathcal{E}}%

\global\long\def\shF{\mathcal{F}}%

\global\long\def\shG{\mathcal{G}}%

\global\long\def\shH{\mathcal{H}}%

\global\long\def\shI{\mathcal{I}}%

\global\long\def\shJ{\mathcal{J}}%

\global\long\def\shK{\mathcal{K}}%

\global\long\def\shL{\mathcal{L}}%

\global\long\def\shM{\mathcal{M}}%

\global\long\def\shN{\mathcal{N}}%

\global\long\def\shO{\mathcal{O}}%

\global\long\def\shP{\mathcal{P}}%

\global\long\def\shQ{\mathcal{Q}}%

\global\long\def\shR{\mathcal{R}}%

\global\long\def\shS{\mathcal{S}}%

\global\long\def\shT{\mathcal{T}}%

\global\long\def\shU{\mathcal{U}}%

\global\long\def\shV{\mathcal{V}}%

\global\long\def\shW{\mathcal{W}}%

\global\long\def\shX{\mathcal{X}}%

\global\long\def\shY{\mathcal{Y}}%

\global\long\def\shZ{\mathcal{Z}}%

\global\long\def\st{\ | \ }%

\global\long\def\stA{\mathcal{A}}%

\global\long\def\stB{\mathcal{B}}%

\global\long\def\stC{\mathcal{C}}%

\global\long\def\stD{\mathcal{D}}%

\global\long\def\stE{\mathcal{E}}%

\global\long\def\stF{\mathcal{F}}%

\global\long\def\stG{\mathcal{G}}%

\global\long\def\stH{\mathcal{H}}%

\global\long\def\stI{\mathcal{I}}%

\global\long\def\stJ{\mathcal{J}}%

\global\long\def\stK{\mathcal{K}}%

\global\long\def\stL{\mathcal{L}}%

\global\long\def\stM{\mathcal{M}}%

\global\long\def\stN{\mathcal{N}}%

\global\long\def\stO{\mathcal{O}}%

\global\long\def\stP{\mathcal{P}}%

\global\long\def\stQ{\mathcal{Q}}%

\global\long\def\stR{\mathcal{R}}%

\global\long\def\stS{\mathcal{S}}%

\global\long\def\stT{\mathcal{T}}%

\global\long\def\stU{\mathcal{U}}%

\global\long\def\stV{\mathcal{V}}%

\global\long\def\stW{\mathcal{W}}%

\global\long\def\stX{\mathcal{X}}%

\global\long\def\stY{\mathcal{Y}}%

\global\long\def\stZ{\mathcal{Z}}%

\global\long\def\then{\ \Longrightarrow\ }%

\global\long\def\L{\textup{L}}%

\global\long\def\l{\textup{l}}%

\renewcommand{\Loc}{\Vect}
\begin{abstract}
Given a quasi-compact category fibered in groupoids $\stX$ and a
monoidal subcategory $\shC$ of its category of locally free sheaves
$\Vect(\stX)$, we are going to introduce the stack of fiber functors
$\Fib_{\stX,\shC}$ with source $\shC$, which comes equipped with
a map $\shP_{\shC}\colon\stX\to\Fib_{\stX,\shC}$ and a functor $\shG\colon\shC\to\Vect(\Fib_{\stX,\shC})$.

If $\shC$ generates $\QCoh(\stX)$ and $\stX$ is an fpqc stack with
quasi-affine diagonal, we show that $\shP_{\shC}\colon\stX\to\Fib_{\stX,\shC}$
is an equivalence, as it happens by Tannaka's reconstruction when
$\stX$ is an affine gerbe over a field. In general, under mild assumption
on $\shC$, e.g. $\shC=\Vect(\stX)$, we show that $\Fib_{\stX,\shC}$
is a quasi-compact fpqc stack with affine diagonal and that the image
$\shG(\shC)$ generates $\QCoh(\Fib_{\stX,\shC})$.
\end{abstract}

\section*{Introduction}

Classical (non-neutral) Tannaka's duality over a field $k$ establishes
a correspondence between $k$-Tannakian categories and affine gerbes
over $k$. More precisely given a $k$-Tannakian category $\shC$
we can define a fibered category $\Pi_{\shC}\colon\Aff/k\to(\text{groupoids)}$
by
\[
\Pi_{\shC}(A)=\{k\text{-linear, strong monoidal and exact functors \ensuremath{\shC}\ensuremath{\ensuremath{\to\Vect}(A)\}}}
\]
where $\Aff/k$ is the category of affine schemes over $k$ and $\Vect(A)=\Vect(\Spec A)$
is the category of finitely presented locally free sheaves on $\Spec A$.
It is then proved that $\Pi_{\shC}$ is an affine gerbe and the functor
\[
\shC\to\Vect(\Pi_{\shC})\comma X\longmapsto(\Gamma\ni\Pi_{\shC}(A)\mapsto\Gamma_{X}\in\Vect(A))
\]
is a $k$-linear and strong monoidal equivalence. Here and in what
follows we think sheaves (e.g locally free, quasi-coherent) on a fibered
category as functors from the fibered category itself to the corresponding
fibered category of sheaves (see \cite[Section 1]{Tonini2014} for
details).

Conversely if $\Delta$ is an affine gerbe over $k$ then $\Vect(\Delta)$
is a $k$-Tannakian category and the functor 
\[
\Delta\to\Pi_{\Vect(\Delta)}\comma(s\colon\Spec A\to\Delta)\longmapsto(s_{|\Vect(\Delta)}^{*}\colon\Vect(\Delta)\to\Vect(A))
\]
is an equivalence. This is called Tannaka's reconstruction while the
previous equivalence is called Tannaka's recognition.

In this paper we want to generalize Tannaka's reconstruction in two
ways: consider more general fibered categories and proper subcategories
of the category of locally free sheaves and introduce generalized
stacks of fiber functors.

Fix a base commutative ring $R$. Consider a category fibered in groupoids
$\stX$ over $\Aff/R$ and a monoidal subcategory $\shC\subseteq\Loc(\stX)$
that we assume is closed under taking duals. Let $\stY$ be another
category fibered in groupoids over $R$. Given an $R$-linear functor
$\Gamma\colon\shC\to\Vect(\stY)$ we say that $\Gamma$ is right exact
if it is exact on right exact sequences (in the ambient category $\QCoh(\stX)$),
whose terms are direct sum of sheaves in $\shC$ (see \ref{def:definition of Fib stacks}
for a precise definition). We define $\Fib_{\stX,\shC}(\stY)$ as
the groupoid of $R$-linear, right exact and strong monoidal functors
$\Gamma\colon\shC\to\Vect(\stY)$. We define the category fibered
in groupoids $\Fib_{\stX,\shC}\colon\Aff/R\to(\text{groupoids})$
by
\[
\Fib_{\stX,\shC}(A)=\Fib_{\stX,\shC}(\Spec A)=\{R\text{-linear, strong monoidal and right exact functors }\text{\ensuremath{\shC}\ensuremath{\ensuremath{\to\Vect}(A)\}}}
\]
We will show that under mild hypothesis on $\shC$ the category $\Fib_{\stX,\shC}$
is automatically fibered in groupoids even without forcing it (see
\ref{thm:when FibX,C has the resolution property}). Notice also that
$\Fib_{\stX,\shC}$ is a stack for the fpqc topology. The analogy
with classical Tannaka's duality is that $\Pi_{\Vect(\Delta)}=\Fib_{\Delta,\Vect(\Delta)}$
for an affine gerbe $\Delta$ over $R=k$ (see \ref{rem:comparison with classical tannakas reconstruction}).

There are two natural functors. An $R$-linear, strong monoidal and
right exact functor
\[
\shG\colon\shC\to\Vect(\Fib_{\stX,\shC})\comma\shG_{\E}(\Gamma\ni\Fib_{\stX,\shC}(A))=\Gamma_{\E}\in\Vect(A)
\]
and a map of fibered categories
\[
\shP_{\shC}\colon\stX\arr\Fib_{\stX,\shC}\comma(s\colon\Spec A\to\stX)\longmapsto(s_{|\shC}^{*}\colon\shC\arr\Loc A)
\]
which generalize the maps defined in the case of affine gerbes.

We address two problems: when $\shP_{\shC}$ is an equivalence and,
if this is not the case, what can we say about the category $\Fib_{\stX,\shC}$.
Notice that the composition $\shP_{\shC}^{*}\circ\shG\colon\shC\to\Vect(\stX)$
is the natural inclusion, so that, if $\shP_{\shC}$ is an equivalence,
then $\shG$ is fully faithful.

By a quasi-compact fibered category we mean a category fibered in
groupoids $\stX$ admitting a map $U\to\stX$ from an affine scheme
which is representable by fpqc covering of algebraic spaces. We say
that a full subcategory $\shD\subseteq\QCoh(\stX)$ generates $\QCoh(\stX)$
if any quasi-coherent sheaf on $\stX$ is a quotient of a direct sum
of sheaves in $\shD$. We prove the following generalization of classical
Tannaka's reconstruction (see \ref{rem:comparison with classical tannakas reconstruction}).

\begin{thma} [\ref{thm: Tannaka reconstrion for Loc X}, \ref{cor:Lurie s point of view}]\label{thma}
Let $\stX$ be a quasi-compact stack over $R$ for the fpqc topology
with quasi-affine diagonal and $\shC\subseteq\Loc\stX$ be a full,
monoidal subcategory with duals generating $\QCoh\stX$. Then $\shP_{\shC}\colon\stX\to\Fib_{\stX,\shC}$
is an equivalence and, if $\stY$ is a category fibered in groupoids
over $R$, the functor 
\[
\Hom(\stY,\stX)\arr\Fib_{\stX,\shC}(\stY)\comma(\stY\arrdi f\stX)\longmapsto f_{|\shC}^{*}\colon\shC\arr\Loc(\stY)
\]
is an equivalence of categories.

\end{thma}

In the case $\shC=\Loc\stX$, the functor $\shP_{\Loc(\stX)}$ has
already been proved to be an equivalence in the neutral case, that
is $\stX=\Bi_{R}G$, where $G$ is a flat and affine group scheme
over $R$ (see \cite[Theorem 1.2]{Broshi2013}, where $R$ is a Dedekind
domain, and \cite[Theorem 1.3.2]{Schappi2011} for general rings $R$),
for particular quotient stacks over a field (see \cite{Savin2006}
and \ref{rem: cases for resolution property}) and for quasi-compact
and quasi-separated schemes (see \cite[Proposition 1.8]{Brandenburg2012}).
More generally, although not explicitly stated elsewhere, for stacks
with the resolution property and with affine diagonal (which is automatic
in the algebraic case, see \cite{Totaro2002}) the fact that $\shP_{\Loc(\stX)}$
is an equivalence is equivalent to the known analogous results where
$\Loc$ is replaced by $\QCoh(-)$ (for those results and other variants
with $\Coh(-)$ or the derived category $D(-)$ see \cite{Lurie2004,Schappi2012,Brandenburg2012,Brandenburg2014,Bhatt2014,Hall2014}):
one can pass from quasi-coherent sheaves to locally free sheaves via
dualizable objects and, for the converse, extend functors from $\Loc(-)$
to $\QCoh(-)$ following the proof of \cite[Corollary 3.2]{Bhatt2014}.
We complete this picture by showing that in general the resolution
property implies the affineness of the diagonal (see \ref{cor:quasi-affine and resolution implies affine}).

The proof we present here does not follow this strategy and one of
the main ingredients is the classification of quasi-compact stacks
whose quasi-coherent sheaves are generated by global sections, called
pseudo-affine. We show that a quasi-compact stack $\stX$ is pseudo-affine
if and only if it is a sheaf with a flat monomorphism $\stX\to\Spec B$
for some $B$ (one can take $B=\Hl^{0}(\odi{\stX})$, see \ref{lem:recognizing affine rings from global sections}).
Moreover we show that a pseudo-affine algebraic stack is quasi-affine
(see \ref{cor:pseudo-affine algebraic is quasi-affine}), which has
already been observed in \cite[Proposition 3.1]{Gross2013}, and that
a quasi-compact flat monomorphism of algebraic stacks is quasi-affine
(see \ref{thm:flat mono are quasi-affine}), which was proved in \cite{Raynaud1967}.

The characterization of pseudo-affine sheaves and Theorem \ref{thma}
are a consequence of the theory developed in \cite{Tonini2014}, where
a correspondence between linear functors $\shC\to\Mod(A)$ and quasi-coherent
sheaves on $\stX\times A$ is discussed. We summarize in Section \ref{sec:Preliminaries}
the results used.

As explained above pseudo-affine sheaves are used in the proof of
Theorem \ref{thma}. On the other hand, another consequence of Theorem
\ref{thma} is a different characterization of pseudo-affine sheaves:
we show that a sheaf $U$ is pseudo-affine if and only if it is the
(sheaf) intersection of quasi-compact open subsets of an affine scheme
$\Spec B$ (one can take $B=\Hl^{0}(\odi U)$, see \ref{prop:arbitrary intersection of quasi-compact open is pseudo-affine}).

A consequence of this new characterization is a partial answer to
our second initial question, the nature of $\Fib_{\stX,\shC}$, and
a partial converse to Theorem \ref{thma}.

\begin{thmb} [\ref{thm:when FibX,C has the resolution property}]
\label{thmc} Let $\stX$ be a quasi-compact fibered category over
$R$ and $\shC\subseteq\Loc\stX$ be a full monoidal subcategory with
duals. If $R$ is not a $\Q$-algebra assume moreover that $\Sym^{n}\E\in\shC$
for all $\E\in\shC$ and $n\in\N$. Then $\Fib_{\stX,\shC}$ is a
quasi-compact fpqc stack with affine diagonal and the subcategory
$\{\shG_{\E}\}_{\E\in\shC}\subseteq\Loc(\Fib_{\stX,\shC})$ generates
$\QCoh(\Fib_{\stX,\shC})$. In particular $\Fib_{\stX,\shC}$ has
the resolution property.

\end{thmb}

We also consider the substack $\Fib_{\stX,\shC}^{f}$ of $\Fib_{\stX,\shC}$
of functors $\Gamma\colon\shC\to\Vect(A)$ such that $f(\E)=\rk\Gamma_{\E}$
for all $\E\in\shC$, where $f$ is a rank function $f\colon\shC\to\N$.
In Theorem \ref{thm:when FibX,C has the resolution property} we prove
that $\Fib_{\stX,\shC}^{f}$ is non empty if and only if there exists
a geometric point $\xi\colon\Spec\Omega\to\stX$ such that $f(\E)=\rk\xi^{*}\E$
for all $\E\in\shC$ and that, in this case, $\Fib_{\stX,\shC}^{f}$
is also a quasi-compact stack with affine diagonal and the resolution
property.

In Theorem \ref{thm:when FibX,C has the resolution property} is also
discussed a criterion to deduce that $\Fib_{\stX,\shC}$ is an algebraic
stack, but we still think this is very unsatisfactory (see \ref{rem:bad algebraic fib}).

We expect that $\shG\colon\shC\to\Vect(\Fib_{\stX,\shC})$ is fully
faithful in general, which would imply that $\stX\to\Fib_{\stX,\Vect(\stX)}$
is universal among maps from $\stX$ to quasi-compact fpqc stacks
with quasi-affine diagonal and the resolution property, but we are
unable to prove it (see \ref{prop:Conjecture implications}).

We conclude the paper by looking at the ``baby case'' $\shC=\{\odi{\stX}\}$.
In this case $\Fib_{\stX,\{\odi{\stX}\}}$ is the pseudo-affine sheaf
intersection of all quasi-compact open subsets of $\Spec\Hl^{0}(\odi{\stX})$
containing the image of $\stX\to\Spec\Hl^{0}(\odi{\stX})$ (see \ref{prop:the simplest category}).
In \ref{exa:counterexample} we show that $\Fib_{\stX,\{\odi{\stX}\}}$
is not algebraic in general, even if $\stX$ is an integral scheme
of finite type over a field.

The outline of the paper is the following. In the first section we
describe the theory of sheafification functors from \cite{Tonini2014}
and some results about frame bundles. In the second section we give
a first description of pseudo-affine stacks and, in the third section,
prove Theorem \ref{thma}. In the fourth section we give a different
description of pseudo-affine stacks, which is then used in the last
section to prove Theorem \ref{thmc}.

\section*{Notation}

In this paper we work over a base commutative, associative ring $R$
with unity. If not stated otherwise a fiber category will be a category
fibered in groupoids over $\Aff/R$, the category of affine schemes
over $\Spec R$, or, equivalently, the opposite of the category of
$R$-algebras. An fpqc stack will be a stack for the fpqc topology.

A map $f\colon\stX'\arr\stX$ of fibered categories is called representable
if for all maps $T\arr\stX$ from an affine scheme (or an algebraic
space) the fiber product $T\times_{\stX}\stX'$ is (equivalent to)
an algebraic space.

Given a flat and affine group scheme $G$ over $R$ we denote by $\Bi_{R}G$
(or simply $\Bi G$ when the base ring is clear) the stack of $G$-torsors
for the \emph{fpqc} topology, which is an fpqc stack with affine diagonal.
When $G\arr\Spec R$ is finitely presented (resp. smooth) then $\Bi_{R}G$
coincides with the stack of $G$-torsors for the fppf (resp. étale)
topology.

By a ``subcategory'' of a given category we mean a ``full subcategory''
if not stated otherwise.

We assume the notations, definition and results of \cite[Section 1]{Tonini2014}.
In particular: the notion of pseudo-algebraic and quasi-compact fibered
categories or maps between them; flat maps of fibered categories;
quasi-coherent sheaves and their functoriality; quasi-coherent sheaves
on a fibered category. In particular $\QCoh(\stX)$ and $\QAlg(\stX)$
denotes the category of quasi-coherent sheaves and quasi-coherent
algebras on $\stX$ respectively.

If $\stY_{i}$ for $i\in I$ is a set of fibered categories over $R$
then $\sqcup_{i}\stY_{i}$ is defined as the fibered category over
$R$ whose objects over an $R$-algebra $A$ are tuples consisting
of a decomposition $\Spec A=\sqcup_{i}U_{i}$ into open and closed
subsets and $y_{i}\in\stY_{i}(U_{i})$. In particular if $\stX$ is
a fibered category with maps $\stY_{i}\to\stX$ then a map
\[
\bigsqcup_{i\in I}\stY_{i}\to\stX
\]
is well defined provided that $\stX$ is a Zariski stack.

If $\E$ is a vector bundle over a category fibered in groupoids $\stX$
then the locus $\stX_{n}\to\stX$ where $\E$ has rank $n$ is an
open and closed immersion. Moreover there is a map $\stX\to\sqcup_{n}\stX_{n}$
(which is an equivalence if $\stX$ is a Zariski stack) and if $\stX$
is quasi-compact, then all $\stX_{n}$ are empty but finitely many.

If $\E$ is a vector bundle on a fibered category $\stX$ we can define
$\det\E$ as 
\[
\det\E=(\bigoplus_{n\in\N}\Lambda^{n}(\E_{|\stX_{n}}))_{|\stX}
\]
which is compatible with the usual notion of determinant for vector
bundles of fixed rank.

\section*{Acknowledgments}

I would like to thank Jarod Alper, Daniel Schäppi, Mattia Talpo and
Angelo Vistoli for the useful conversations I had with them and all
the suggestions they gave me. Moreover I would like to thank David
Rydh for sharing many useful references and explaining me how the
non-neutral form of Tannaka's reconstruction proved in this paper
fits in the vast literature about Tannaka's duality.

\section{\label{sec:Preliminaries}Preliminaries}

We recall some basic definitions and set up some notations. We fix
a base ring $R$.

\subsection{Monoidal functors and their sheafification}
\begin{defn}
\label{def: pseudo monoidal commutativi associative} Let $\catC$
and $\catD$ be $R$-linear symmetric monoidal categories. A (contravariant)
\emph{pseudo-monoidal }functor $\Omega\colon\catC\to\catD$ is an
$R$-linear (and contravariant) functor together with a natural transformation
\[
\iota_{V,W}^{\Omega}\colon\Omega_{V}\otimes\Omega_{W}\to\Omega_{V\otimes W}\text{ for }V,W\in\catC
\]
A (contravariant) pseudo-monoidal functor $\Omega\colon\catC\to\catD$
is

\begin{enumerate}
\item \emph{\label{enu:commutative for functors} symmetric} or commutative
if for all $V,W\in\catC$ the following diagram is commutative$$
\begin{tikzpicture}[xscale=2.7,yscale=-1.2]     \node (A0_1) at (1, 0) {$\Omega_V\otimes \Omega_W$};     \node (A0_2) at (2, 0) {$\Omega_{V\otimes W}$};     \node (A1_1) at (1, 1) {$\Omega_W\otimes \Omega_V$};     \node (A1_2) at (2, 1) {$\Omega_{W\otimes V}$};     \path (A0_1) edge [->]node [auto] {$\scriptstyle{\iota_{V,W}^\Omega}$} (A0_2);     \path (A0_2) edge [->]node [auto] {$\scriptstyle{}$} (A1_2);     \path (A1_1) edge [->]node [auto] {$\scriptstyle{\iota_{W,V}^\Omega}$} (A1_2);     \path (A0_1) edge [->]node [auto] {$\scriptstyle{}$} (A1_1);   \end{tikzpicture}
$$where the vertical arrows are the obvious isomorphisms;
\item \emph{\label{enu:associative for functors} associative} if for all
$V,W,Z\in\catC$ the following diagram is commutative $$
 \begin{tikzpicture}[xscale=4,yscale=-1.2]     \node (A0_0) at (0, 0) {$\Omega_V\otimes \Omega_W\otimes \Omega_Z$};     \node (A0_1) at (1, 0) {$\Omega_{V\otimes W}\otimes \Omega_Z$};     \node (A1_0) at (0, 1) {$\Omega_V\otimes \Omega_{W\otimes Z}$};     \node (A1_1) at (1, 1) {$\Omega_{V\otimes W\otimes Z}$};     \path (A0_0) edge [->]node [auto] {$\scriptstyle{\iota_{V,W}^\Omega \otimes \id}$} (A0_1);     \path (A1_0) edge [->]node [auto] {$\scriptstyle{\iota_{V,W\otimes Z}^\Omega}$} (A1_1);     \path (A0_0) edge [->]node [auto] {$\scriptstyle{\id\otimes \iota_{W,Z}^\Omega}$} (A1_0);     \path (A0_1) edge [->]node [auto] {$\scriptstyle{\iota_{V\otimes W,Z}^\Omega}$} (A1_1);   \end{tikzpicture}
$$
\end{enumerate}
If $I$ and $J$ are the unit objects of $\catC$ and $\catD$ respectively,
a unity for $\Omega$ is a morphism $1\colon J\to\Omega_{I}$ such
that, for all $V\in\catC$, the compositions
\[
\Omega_{V}\otimes J\to\Omega_{V}\otimes\Omega_{I}\to\Omega_{V\otimes I}\to\Omega_{V}\text{ and }J\otimes\Omega_{V}\to\Omega_{I}\otimes\Omega_{V}\to\Omega_{I\otimes V}\to\Omega_{V}
\]
coincide with the natural isomorphisms $\Omega_{V}\otimes J\to\Omega_{V}$
and $J\otimes\Omega_{V}\to\Omega_{V}$ respectively.

A (contravariant)\emph{ monoidal} functor $\Omega\colon\catC\to\catD$
is a symmetric and associative pseudo-monoidal (contravariant) functor
with a unity $1$. A (contravariant) \emph{strong} monoidal functor
$\Omega\colon\catC\to\catD$ is a (contravariant) monoidal functors
such that all the maps $\iota_{V,W}^{\Omega}$and $1\colon J\to\Omega_{I}$
are isomorphisms.
\end{defn}

A morphism of pseudo-monoidal functors $(\Omega,\iota^{\Omega})\to(\Gamma,\iota^{\Gamma})$,
called a monoidal morphism or transformation, is a natural transformation
$\Omega\to\Gamma$ which commutes with the monoidal structures $\iota^{*}$.
A morphism of monoidal functors is a monoidal transformation preserving
the unities.
\begin{defn}
Let $\stX$ be a fibered category over $R$ and $\shC$ a full subcategory
of $\QCoh(\stX)$. We say that $\shC$ generates $\QCoh(\stX)$ if
all quasi-coherent sheaves on $\stX$ are a quotient of a direct sum
of sheaves in $\shC$. We say that $\stX$ has the \emph{resolution
property} if $\Loc(\stX)$ generates $\QCoh(\stX)$.
\end{defn}

We will consider only fpqc stacks with quasi-affine diagonal, for
instance algebraic stacks with quasi-affine diagonal (see \cite[Corollary 10.7]{Laumon1999})
and quasi-separated schemes. This is because resolution property is
somehow meaningless for other stacks, see for instance Remark $(1)$
in the introduction of \cite{Totaro2002}.
\begin{defn}
\label{def:basic categories} \cite[Def 2.1]{Tonini2014} Let $\stX$
be a fibered category over $R$, $A$ an $R$-algebra and $\shD$
a subcategory $\shD\subseteq\QCoh(\stX)$.

We define $\L_{R}(\shD,A)$ as the category of contravariant $R$-linear
functors $\Gamma\colon\shD\to\Mod A$ and natural transformations
as arrows. If $\shD$ is a monoidal subcategory of $\QCoh\stX$, that
is a subcategory such that $\odi{\stX}\in\shD$ and for all $\E,\E'\in\shD$
we have $\E\otimes\E'\in\shD$, we also define the category $\ML_{R}(\shD,A)$
whose objects are $\Gamma\in\L_{R}(\shD,A)$ with a monoidal structure.

We denote by $\shD^{\oplus}$ the full subcategory of $\QCoh(\stX)$
containing all the finite direct sums of elements in $\shD$. Notice
that a (contravariant) $R$-linear functor from $\shD$ to an $R$-linear
and additive category extends uniquely to $\shD^{\oplus}$ (see \cite[Prop 2.16]{Tonini2014}).
We will denote this extension by the same symbol. In other words we
will evaluate a linear functor with source $\shD$ also on objects
and maps of $\shD^{\oplus}$.
\end{defn}

\begin{defn}
\label{def:exactness} A \emph{finite test sequence }for $\shD$ is
an exact sequence in $\QCoh(\stX)$ of the form
\[
\E'\to\E\to0\text{ or }\E''\to\E'\to\E\to0\text{ with }\E'',\E',\E\in\shD^{\oplus}
\]
If $\shD\subseteq\Loc(\stX)$, $\stX$ is quasi-compact and $\stY$
is another fiber category over $R$ we say that a (contravariant)
$R$-linear functor $\Gamma\colon\shD\to\QCoh(\stY)$ is right (resp.
left) exact if it is exact on all test sequences in $\shD$. We define
$\Lex_{R}(\shD,A)$ (resp. $\MLex_{R}(\shD,A)$) as the subcategory
of $\L_{R}(\shD,A)$ (resp. $\ML_{R}(\shD,A)$) of left exact functors.

{} We define $\stX_{A}=\stX\times_{R}A$ and
\[
\Omega^{*}\colon\QCoh\stX_{A}\to\L_{R}(\shD,A)\comma\Omega_{\shH}^{\shG}=\Hom_{\stX}(\pi^{*}\shH,\shG)
\]
where $\pi\colon\stX_{A}\to\stX$ is the projection.
\end{defn}

\begin{rem}
The definition of test sequence and therefore of the categories $\Lex_{R}(\shD,A)$
and $\MLex_{R}(\shD,A)$ it is not completely equivalent to the one
in \cite[Def 4.3]{Tonini2014}. As discussed in the proof of the next
Theorem the two notions agree when $\shD$ generates $\QCoh(\stX)$:
functors left exact in the sense of \cite{Tonini2014} are automatically
left exact in our sense. In \cite{Tonini2014} we were looking for
a minimal collection of test sequences which make results like the
one below true under the assumption that $\shD$ generates $\QCoh(\stX)$.
In this paper instead we will deal with more general subcategories
$\shD$ and therefore this notion of test sequence seems more precise.
\end{rem}

\begin{thm}
\label{thm:main thm main paper} Let $\stX$ be a quasi-compact fibered
category over $R$ and $\shD\subseteq\Loc(\stX)$ be a full subcategory
that generates $\QCoh(\stX)$. Then $\Omega^{*}\colon\QCoh\stX_{A}\to\L_{R}(\shD,A)$
is fully faithful with essential image $\Lex_{R}(\shD,A)$ and it
has an exact left adjoint $\shF_{*,\shD}\colon\L_{R}(\shD,A)\to\QCoh\stX_{A}$.
In particular $\Omega_{*}$ and $(\shF_{*,\shD})_{|\Lex_{R}(\shD,A)}$
are quasi-inverses of each other.

Assume that $\shD$ is a monoidal subcategory of $\Loc(\stX)$. Then
the functor $\Omega^{*}$ and $\shF_{*,\shD}$ extends to adjoint
functors $\Omega^{*}\colon\QAlg\stX_{A}\to\ML_{R}(\shD,A)$ and $\alA_{*,\shC}\colon\ML_{R}(\shD,A)\to\QAlg\stX_{A}$.
Moreover $\Omega^{*}\colon\QAlg(\stX_{A})\to\MLex_{R}(\shD,A)$ is
an equivalence and $(\alA_{*,\shC})_{|\MLex_{R}(\shD,A)}$ is a quasi-inverse.
\end{thm}

\begin{proof}
By \cite[Theorem A]{Tonini2014}, \cite[Theorem B]{Tonini2014} and
\cite[Theorem 6.8]{Tonini2014} the same result works for the notion
of test sequence given in \cite[Def 2.1]{Tonini2014}. On the other
hand by \cite[Prop 4.5]{Tonini2014} functors in the image of $\Omega^{*}\colon\QCoh\stX_{A}\to\L_{R}(\shD,A)$
are automatically left exact in our stronger sense.
\end{proof}
\begin{rem}
If $\stX$ is a quasi-compact fibered category over $R$ and $\shC\subseteq\Vect(\stX)$
is a full (monoidal) subcategory there always exists a full (monoidal)
subcategory $\tilde{\shC}\subseteq\shC$ such that $\widetilde{\shC}$
is small, that is the class $\text{ob}(\tilde{\shC})$ is a set, and
$\widetilde{\shC}\to\shC$ is an equivalence. In particular in this
case the restriction along $\widetilde{\shC}\to\shC$ induces an equivalence
for all the categories of linear (monoidal) functors we have considered.
Thus, when $\stX$ is a quasi-compact fibered category over $R$,
we will tacitly assume that the category $\shC$ is small.

To see how to find $\widetilde{\shC}$ we use the following argument.
The category $\shC$ is essentially small because $\stX$ is quasi-compact
(see \cite[Prop 1.8]{Tonini2014}), so we can start taking $\shR_{0}\subseteq\text{ob}(\shC)$
so that any sheaf in $\shC$ is isomorphic to a sheaf in $\shR_{0}$.
If we don't need duals or monoidal structure we can stop here and
consider as $\widetilde{\shC}$ the full subcategory with objects
in $\shR_{0}$.

If $\shC$ is a submonoidal category, this choice is not enough, since
we don't want to change the $\otimes$. Assume also $\odi{\stX}\in\shR_{0}$.
At this point we denote by $\shR$ the set of objects of $\shC$ that
can be written using finitely many times $\otimes$ and $-^{\vee}$,
starting from objects in $\shR_{0}$. Setting as $\tilde{\shC}$ the
full subcategory of $\shC$ with objects in $\shR$ makes the trick.
\end{rem}

\subsection{\label{subsec:Generalized-frame-bundles} Generalized frame bundles}

In this section we study frame bundles of general vector bundles.
We fix a category $\stX$ fibered in groupoids over $R$.
\begin{defn}
\label{def:profinite stuff} Given a subset $K\subseteq\N$ we set
\[
D(K)=\bigsqcup_{n\in K}\Spec R
\]
and denote by $Q(K)$ the vector bundle on $D(K)$ which is free of
rank $n$ over the copy of $\Spec R$ corresponding to $n\in K$.
More generally given a collection $K_{*}\colon I\to\shP(\N)$, where
$\shP(\N)$ is the set of subsets of $\N$, we set $D(K_{*})=\prod_{i}D(K_{i})$
as a sheaf over $R$.

Let $\E\in\Vect(\stX)$ be a locally free sheaf on $\stX$. We set
\[
\text{ranks}(\E)=\{n\in\N\st\exists s\colon\Spec L\to\stX\text{ with }\rk s^{*}\E=n\}\comma D(\E)=D(\text{ranks}(\E))\comma Q(\E)=Q(\text{ranks}(\E))
\]

There is a map $r_{\E}\colon\stX\to D(\E)$, the rank function, and
we set
\[
\Fr(\E)=\Isosh_{\stX}(\E,r_{\E}^{*}Q(\E))\to\stX
\]
If $I$ is a set and $\E_{*}\colon I\to\Vect(\stX)$ is a collection
of vector bundles we set $D(\E_{*})=D(\text{ranks}(\E_{*}))=\prod_{i}D(\E_{i})$
and $\Fr(\E_{*})=\prod_{i}\Fr(\E_{i})$, where the products are taken
over $R$ and $\stX$ respectively.
\end{defn}

\begin{rem}
\label{rem:mapping to D(E*)} The scheme $D(K)$, thought of as a
sheaf, is the constant sheaf associated with $K$. Instead $D(K_{*})$
as sheaf $\Aff/R\to\sets$ is
\[
D(K_{*})(U)=\{s\colon U\to\prod_{i\in I}K_{i}\text{ with locally constant factors }U\to K_{i}\}
\]
The map $\stX\to D(\E_{*})$ maps $\xi\colon U\to\stX$ to the map
$\prod_{i}\rk(\xi^{*}\E_{i})\colon U\to\prod_{i\in I}\text{ranks}(\E_{i})$.
\end{rem}

\begin{rem}
\label{rem:connected components vector affine} The scheme $D(K)$
is affine if and only if $K$ is finite. If $K_{*}\colon I\to\shP(\N)$
is a collection such that all the $K_{i}$ are finite, then $D(K_{*})\to\Spec R$
is also affine and faithfully flat. This is the case, for instance,
if $\stX$ is quasi-compact and $K_{*}=\text{ranks}(\E_{*})$. Otherwise
$D(K_{*})$ is just thought of as a sheaf over $R$.
\end{rem}

\begin{lem}
\label{lem:rank loci} Let $K_{*}\colon I\to\shP(\N)$ be a map and
$n_{i}\in K_{i}$ for all $i$. Then the map $\Spec R\to D(K_{*})$
given by the constant function $(n_{i})_{i}\colon\Spec R\to\prod_{i\in I}K_{i}$
is a closed immersion and representable by localizations of affine
schemes. In particular if $\E_{*}\colon I\to\Vect(\stX)$ is a collection
of vector bundles then the locus $\stX'\to\stX$ where the sheaves
$\E_{i}$ have rank $n_{i}$ enjoys the same property.
\end{lem}

\begin{proof}
Let $\Spec B\to D(K_{*})$ be any map, given by locally constant functions
$h_{i}\colon\Spec B\to K_{i}$. Set $Z=\Spec B\times_{D(K_{*})}\Spec R$.
We have to show that $Z\to\Spec B$ is a closed immersion and a localization.
By definition of fiber product we see that $Z$ is the locus in $\Spec B$
where the $h_{i}$ are constant and equal to $n_{i}$. In other words
if $U_{i}=h_{i}^{-1}(n_{i})$ then $Z=\cap_{i}U_{i}$. Notice that
any intersections of the $U_{i}$ is an open and closed subscheme
of $\Spec B$ and, in particular, a localization. Since $Z$ can be
written as a filtered intersection of open and closed subsets of $\Spec B$,
it is elementary to check that $Z\to\Spec B$ is both a closed immersion
and a localization.
\end{proof}
\begin{prop}
\label{prop:generalized GLn} Let $K\subseteq\N$. The stack $\Bi_{D(K)}\GL(Q(K))\to D(K)$
is equivalent to the stack over $D(K)$ whose fiber over a $\Spec B\to D(K)$,
given by $h\colon\Spec B\to K$, is the groupoid of locally free sheaves
on $\Spec B$ that have rank $n$ on $h^{-1}(n)$ for all $n\in K$.

Let $K_{*}\colon I\to\shP(\N)$ and 
\[
G=\prod_{i}\GL(Q(K_{i}))=\prod_{i}(\GL(Q(K_{i}))\times_{D(K_{i})}D(K_{*}))\to D(K_{*})
\]
where the products are taken over $R$ and $D(K_{*})$ respectively.
Then $G\to D(K_{*})$ is an affine and faithfully flat relative group
scheme and the stack $\Bi_{D(K_{*})}G\to D(K_{*})$ is equivalent
to the stack over $D(K_{*})$ whose fiber over a $\Spec B\to D(K_{*})$,
given by $h_{i}\colon\Spec B\to K_{i}$, is the groupoid of $(\shG_{i})_{i}$
where $\shG_{i}$ is a locally free sheaf on $\Spec B$ that have
rank $n$ on $h^{-1}(n)$ for all $n\in K_{i}$.
\end{prop}

\begin{proof}
The first statement reduced easily to the case of $\GL_{n}$. For
the second part set $G_{i}=\GL(Q(K_{i}))\times_{D(K_{i})}D(K_{*})$.
Everything follows because $G=\prod_{i}G_{i}$ and therefore 
\[
\Bi_{D(K_{*})}G=\prod_{i}\Bi_{D(K_{*})}G_{i}
\]
where the products are taken over $D(K_{*})$.
\end{proof}
\begin{lem}
\label{lem:frames} Let $\E\in\Vect(\stX)$. Then there is a Cartesian
diagram   \[   \begin{tikzpicture}[xscale=2.7,yscale=-1.2]     \node (A0_0) at (0, 0) {$\Fr(\E)$};     \node (A0_1) at (1, 0) {$D(\E)$};     \node (A1_0) at (0, 1) {$\stX$};     \node (A1_1) at (1, 1) {$\Bi_{D(\E)} (\GL(Q(\E)))$};     \path (A0_0) edge [->]node [auto] {$\scriptstyle{}$} (A0_1);     \path (A1_0) edge [->]node [auto] {$\scriptstyle{}$} (A1_1);     \path (A0_1) edge [->]node [auto] {$\scriptstyle{}$} (A1_1);     \path (A0_0) edge [->]node [auto] {$\scriptstyle{}$} (A1_0);   \end{tikzpicture}   \] 
where $\stX\to\Bi_{D(\E)}\GL(Q(\E))$ is induced by $\E$ as in \ref{prop:generalized GLn}.
In particular $\Fr(\E)\to\stX$ is a torsor under the affine group
scheme $\GL(Q(\E))\to D(\E)$ and it is affine, smooth and surjective.

Let $\E_{*}\colon I\to\Vect(\stX)$ be a collection of sheaves and
consider $G=\prod_{i}\GL(Q(\E_{i}))\to\prod_{i}D(Q(\E_{i}))=D(\E_{*})$.
Then there is a Cartesian diagram   \[   \begin{tikzpicture}[xscale=2.7,yscale=-1.2]     \node (A0_0) at (0, 0) {$\Fr(\E_*)$};     \node (A0_1) at (1, 0) {$D(\E_*)$};     \node (A1_0) at (0, 1) {$\stX$};     \node (A1_1) at (1, 1) {$\Bi_{D(\E_*)} (G)$};     \path (A0_0) edge [->]node [auto] {$\scriptstyle{}$} (A0_1);     \path (A1_0) edge [->]node [auto] {$\scriptstyle{}$} (A1_1);     \path (A0_1) edge [->]node [auto] {$\scriptstyle{}$} (A1_1);     \path (A0_0) edge [->]node [auto] {$\scriptstyle{}$} (A1_0);   \end{tikzpicture}   \] 
where $\stX\to\Bi_{D(\E_{*})}G$ is induced by the $(\E_{i})_{i}$
as in \ref{prop:generalized GLn}. In particular $\Fr(\E_{*})\to\stX$
is affine and faithfully flat as well.
\end{lem}

\begin{proof}
It follows from the definition of $\Fr$ and the description in \ref{prop:generalized GLn}.
\end{proof}
\begin{lem}
\label{lem:det and sym for locally free sheaves} Let $\E$ be a locally
free sheaf on $\stX$ of rank $n$ and set $V$ for the free $R$-module
of rank $n$ with basis $v_{1},\dots,v_{n}$. Then the map 
\[
\gamma\colon\det\E\to\Sym^{n}(V\otimes\E)\comma\gamma(x_{1}\wedge\cdots\wedge x_{n})=\sum_{\sigma\in S_{n}}(-1)^{\sigma}\prod_{k}(v_{\sigma(k)}\otimes x_{k})
\]
is well defined and has locally free cokernel. In particular its dual
$\Sym^{n}(V\otimes\E)^{\vee}\to\det\E^{\vee}$ is surjective. Moreover,
if $\E$ is free with basis $e_{1},\dots,e_{n}$then $e_{1}\wedge\cdots\wedge e_{n}\in\det\E$
is sent to the determinant of the matrix $(v_{i}\otimes e_{j})_{ij}$,
which lies in $\Sym^{n}(V\otimes\E)$.
\end{lem}

\begin{proof}
Consider the lift $\tilde{\gamma}\colon\E^{\otimes n}\to(V\otimes\E)^{\otimes n}\to\Sym^{n}(V\otimes\E)$
of $\gamma$. This is certainly well defined. In order to show that
it factors through the quotient $\det\E$, so that $\gamma$ would
be well defined, we can assume that $\E$ is free with basis $e_{1},\dots,e_{n}$.
Set $y_{ij}=v_{i}\otimes e_{j}\in V\otimes\E$ and $e_{\omega}=e_{\omega(1)}\otimes\cdots\otimes e_{\omega(n)}\in\E^{\otimes n}$
for any map $\omega\colon\{1,\dots,n\}\to\{1,\dots,n\}$. We have
\[
\tilde{\gamma}(e_{\omega})=\sum_{\sigma\in S_{n}}(-1)^{\sigma}\prod_{k}y_{\sigma(k)\omega(k)}=\det(y_{i,\omega(j)})
\]
It follows that $\widetilde{\gamma}(e_{\omega})=0$ if $\omega$ is
not a permutation, that $\widetilde{\gamma}(e_{\tau\omega})=-\widetilde{\gamma}(e_{\omega})$
if $\tau$ is a transposition and that $\widetilde{\gamma}(e_{\id})=\det(v_{i}\otimes e_{j})$,
as required.

We now show that the cokernel of $\gamma$ is locally free. This is
equivalent to the surjectivity of $\Sym^{n}(V\otimes\E)^{\vee}\to\det\E^{\vee}$
and it can be check locally. Consider $e\colon\Sym^{n}(V\otimes\E)\to R$
given by $e(y_{ij})=\delta_{ij}$, the Kronecker symbol. We see that
the composition $\det\E\to\Sym^{n}(V\otimes\E)\arrdi eR$ maps $e_{1}\wedge\cdots\wedge e_{n}\to\det(I)=1$,
where $I=(\delta_{ij})$ is the identity matrix.
\end{proof}
\begin{lem}
\label{lem:standard quotients} Let $\E$ be a vector bundle on a
fibered category $\stX$ over $R$ then:
\begin{enumerate}
\item the sheaf $\det\E$ is a quotient of a direct sum of the $\E{}^{\otimes m}$
for $m\in\N$;
\item if $\stX$ is defined over $\Q$ then $(\Sym^{q}\E)^{\vee}$ is a
quotient of $(\E^{\vee})^{\otimes q}$.
\item the sheaf $\Sym^{q}(\E\otimes M)^{\vee}$, for a locally free $R$-module
$M$, is a quotient of a direct sum of finite tensor products of the
$\Sym^{a}(\E)^{\vee}$ for $a\in\N$.
\end{enumerate}
\end{lem}

\begin{proof}
We can assume that $\stX$ is a Zariski stack. Set $I=\text{ranks}(\E)$
and denote by $\stX_{n}$ for $n\in I$ the locus in $\stX$ where
$\E$ has rank $n$, so that $\sqcup_{n}\stX_{n}=\stX$.

$1)$. There are surjective maps

\[
\bigoplus_{n\in I}\E^{\otimes n}\to\bigoplus_{n\in I}(\E^{\otimes n})_{|\stX_{n}}\to\bigoplus_{n\in I}\det(\E_{|\stX_{n}})\simeq\det\E
\]
 $2)$ There is a surjective map 
\[
(\E^{\vee})^{\otimes q}\to\Sym^{q}(\E^{\vee})\simeq(\Sym^{q}\E)^{\vee}
\]
$3)$ Write $M\oplus N=R^{l}$. We use the formula 
\[
\Sym^{q}(A\oplus B)=\bigoplus_{u+v=q}\Sym^{u}(A)\otimes\Sym^{v}(B)
\]
We have
\[
\Sym^{q}(\E\otimes R^{l})=\Sym^{q}(\E\otimes M)\oplus P\then\Sym^{q}(\E\otimes R^{l})^{\vee}\twoheadrightarrow\Sym^{q}(\E\otimes M)^{\vee}
\]
Moreover $\Sym^{q}(\E\otimes R^{l})^{\vee}$ can be written as a quotient
of a direct sum of finite tensor products of the$\Sym^{a}(\E)^{\vee}$
for $a\in\N$. 
\end{proof}
\begin{lem}
\label{lem:generators for GLn} Let $I$ be a set and $\E_{*}\colon I\to\Vect(R)$
be any map. Set

\[
\GL(\E_{*})=\prod_{i\in I}\GL(\E_{i})\to\Spec R
\]
and $\shF_{i}$ for the locally free sheaf on $\Bi\GL(\E_{*})$ coming
from the universal one on $\Bi\GL(\E_{i})$, that is the $R$-module
$\E_{i}$ with the natural action of $\GL(\E_{i})$. Consider the
following sets: 
\[
\shR_{1}=\{\det\shF_{i},(\Sym^{m}\shF_{i})^{\vee}\}_{i\in I,m\in\N}\comma\shR_{2}=\{\shF_{i},(\Sym^{m}\shF_{i})^{\vee}\}_{i\in I,m\in\N}\comma\shR_{3}=\{\shF_{i},\shF_{i}^{\vee}\}_{i\in I}
\]
 Then the subcategory of $\Loc(\Bi\GL(\E_{*}))$ consisting of all
tensor products of sheaves in $\shR_{1}$ (resp. $\shR_{2})$ generates
$\QCoh(\Bi\GL(\E_{*}))$.

If $R$ is a $\Q$-algebra then we can replace $(\Sym^{m}\shF_{i})^{\vee}$
by $\shF_{i}^{\vee}$ in the statement above.
\end{lem}

\begin{proof}
Denote by $\shC\subseteq\QCoh(\Bi\GL(\E_{*}))$ the full subcategory
generated by the sheaves $\shR_{i}$ in the statement (here $i$ can
be $1,2,3$ depending on the hypothesis). We have to show that $\shC$
generates $\QCoh(\Bi\GL(\E_{*}))$, so that, a posteriori, $\shC=\QCoh(\Bi\GL(\E_{*}))$.
We claim that $R[\GL(\E_{*})]$ is a direct limit of representations
whose dual belongs to $\shC$. By \cite[Prop 8.2]{Tonini2014} this
will ends the proof.

The representation $R[\GL(\E_{*})]$ is a direct limit of tensor products
of the regular representations $R[\GL(\E_{i})]$ for $i\in I$. This
allows us to reduce to the case of $\GL(\E)$ when $\E$ is a vector
bundle on $\E$ (i.e. $I$ has one element). Call $\shF$ the universal
locally free sheaf on $\Bi\GL(\E)$, that is $\shF=\E$ with the action
of $\GL(\E)$. Recall that for any locally free representation $\shH$
of $R[\GL(\E)]$ there is a canonical isomorphism of $R$-modules
$\Hom^{\GL(\E)}(\shH,R[\GL(\E)])\simeq\shH^{\vee}$ (see \cite[Rem 8.1]{Tonini2014})
So there are equivariant maps 
\[
\shF\otimes\Hom^{\GL(\E)}(\shF,R[\GL(\E)])\simeq\shF\otimes\E^{\vee}\to R[\GL(\E)]
\]
 and, for $m\in\N$,
\[
(\det\shF)^{-\otimes m}\otimes\Hom^{\GL(\E)}((\det\shF)^{-\otimes m},R[\GL(\E)])\simeq(\det\shF)^{-\otimes m}\otimes(\det\E)^{\otimes m}\to R[\GL(\E)]
\]
Consider the equivariant maps
\[
A_{m}=[\bigoplus_{q=0}^{2mn}\Sym^{q}(\shF\otimes\E^{\vee})]\otimes[(\det\shF)^{-\otimes m}\otimes(\det\E)^{\otimes m}]\arr R[\GL(\E)]\otimes R[\GL(\E)]\to R[\GL(\E)]
\]
We claim that those maps are injective and their images form an increasing
sequence of sub representations saturating $R[\GL(\E)]$. This statement
is local, so we can assume $\E=\odi{\stX}^{n}$, so that $\GL(\E)=\GL_{n}$.
As usual we can write $R[\GL_{n}]=R[X_{uv}]_{\det}$ for $1\leq u,v\leq n$,
where $\det$ is the determinant polynomial. In this case $\shF\otimes\E^{\vee}\to R[\GL_{n}]$
is an isomorphism onto the $R$-submodule generated by all the $X_{u,v}$,
while $(\det\shF)^{-\otimes m}\otimes(\det\E)^{\otimes m}\to R[\GL(\E)]$
is an isomorphism onto the $R$-submodule generated by $\det^{-m}$.
Thus $A_{m}\to R[\GL_{n}]$ is an isomorphism onto the set of fractions
$f/\det^{m}$ where $f$ is a polynomial of degree less or equal to
$2nm$. It is now easy to see that $\cup_{m}A_{m}=R[\GL_{n}]$.

We come back to the general setting. We need to show that $A_{m}^{\vee}\in\shC$
for all $m$. The last statement follows from the first thanks to
\ref{lem:standard quotients}, $2)$. Since $\shC$ is closed under
tensor product and direct sum, we have to show that 
\[
(\Sym^{q}(\shF\otimes\E^{\vee}))^{\vee}\text{ for }q>0,(\det\shF\otimes\det\E{}^{-1})\in\shC
\]
For the first sheaf it follows from \ref{lem:standard quotients},
$3)$. For the second, since $\det\E^{-1}$ is a quotient of a free
$R$-module, it is enough to show that $\det\shF\in\shC$. For $\shR_{1}$
is clear. For $\shR_{2}$ we use \ref{lem:standard quotients}, $1)$.
\end{proof}

\section{Pseudo-affine stacks}

One of the key points in the proof of Theorem \ref{thma} is a characterization
of the following stacks.
\begin{defn}
A \emph{pseudo-affine} stack is a quasi-compact fpqc stack with quasi-affine
diagonal such that all quasi-coherent sheaves on it are generated
by global sections. In other words a pseudo-affine stack is a quasi-compact
fpqc stack $\stX$ such that $\{\odi{\stX}\}$ generates $\QCoh(\stX)$.

A map $f\colon\stY'\arr\stY$ of fibered categories is called \emph{pseudo-affine
}if for all maps $T\arr\stY$ from an affine scheme the fiber product
$T\times_{\stY}\stY'$ is pseudo-affine.
\end{defn}

A quasi-affine scheme is pseudo-affine. We will prove a pseudo-affine
stack which is algebraic is quasi-affine. This result appeared before
in \cite[Proposition 3.1]{Gross2013}. We will show that pseudo-affine
stacks are indeed just arbitrary intersection of quasi-compact open
subsets of an affine scheme (see \ref{prop:arbitrary intersection of quasi-compact open is pseudo-affine}).
In general a pseudo-affine sheaf is not quasi-affine. An example is
the sheaf intersection of all the complement of closed points in $\Spec k[x,y]$,
where $k$ is a field (see \ref{exa:counterexample pseudo-affine}).

We start with a first characterization of pseudo-affine stacks.
\begin{prop}
\label{lem:recognizing affine rings from global sections} Let $\stX\arrdi{\pi}\Spec R$
be a quasi-compact fpqc stack with quasi-affine diagonal. Then the
following conditions are equivalent:

\begin{enumerate}
\item the stack $\stX$ is pseudo-affine;
\item the map $\pi^{*}\pi_{*}\shF\arr\shF$ is surjective for all $\shF\in\QCoh\stX$;
\item the stack $\stX$ is equivalent to a sheaf and there exists a flat
monomorphism $\stX\arr\Spec B$, where $B$ is a ring.
\end{enumerate}
In this case the map $p\colon\stX\arr\Spec\Hl^{0}(\odi{\stX})$ is
a flat monomorphism, $p_{*}\colon\QCoh\stX\arr\Mod\Hl^{0}(\stX)$
is fully faithful and $p^{*}p_{*}\simeq\id$. Moreover if $\stX\times_{\Hl^{0}(\odi{\stX})}k\neq\emptyset$
for all geometric points $\Spec k\arr\Spec\Hl^{0}(\odi{\stX})$ then
$p$ is an isomorphism.
\end{prop}

\begin{proof}
$2)\then1)$. Given $\shF\in\QCoh\stX$, take a surjective map $R^{(I)}\arr\pi_{*}\shF$.
In this case the composition
\[
\odi{\stX}^{(I)}\simeq\pi^{*}R^{(I)}\arr\pi^{*}\pi_{*}\shF\arr\shF
\]
 is surjective.

$1)\then2)$. A sheaf $\shF\in\QCoh\stX$ is generated by global sections
and the image of $\pi^{*}\pi_{*}\shF\arr\shF$ contains all of them.

$3)\then1)$. Denote by $p\colon\stX\arr\Spec B$ the flat monomorphism.
We are going to show that $\delta_{\shF}\colon p^{*}p_{*}\shF\arr\shF$
is an isomorphism for all $\shF\in\QCoh\stX$. Arguing as in $2)\then1)$
this will conclude the proof. By hypothesis there exists a representable
fpqc covering $h\colon\Spec C\arr\stX$ and we must prove that $h^{*}\delta_{\shF}$
is an isomorphism. Let $f=ph\colon\Spec C\arr\Spec B$ be the composition
and consider the commutative diagram   \[   \begin{tikzpicture}[xscale=2.0,yscale=-1.0]     \node (A0_0) at (0, 0) {$\Spec C$};     \node (A0_1) at (1, 0) {$\stX\times_B C$};     \node (A0_2) at (2, 0) {$\Spec C$};     \node (A1_1) at (1, 1) {$\stX$};     \node (A1_2) at (2, 1) {$\Spec B$};     \path (A0_0) edge [->]node [auto] {$\scriptstyle{s}$} (A0_1);     \path (A0_1) edge [->]node [auto] {$\scriptstyle{\alpha}$} (A0_2);     \path (A1_1) edge [->]node [auto] {$\scriptstyle{p}$} (A1_2);     \path (A0_2) edge [->]node [auto] {$\scriptstyle{f}$} (A1_2);     \path (A0_0) edge [->]node [auto] {$\scriptstyle{h}$} (A1_1);     \path (A0_1) edge [->]node [auto] {$\scriptstyle{t}$} (A1_1);   \end{tikzpicture}   \] 
Since $\alpha$ is a monomorphism with a section, $\alpha$ and $s$
are inverses of each other. The pullback along $t$ of the adjoint
map $\delta_{\shF}\colon p^{*}p_{*}\shF\to\shF$ is 
\[
t^{*}p^{*}p_{*}\shF=\alpha^{*}(f^{*}p_{*}\shF)\to\alpha^{*}\alpha_{*}(t^{*}\shF)\to t^{*}\shF
\]
 The first map is an isomorphism because $f$ is flat and therefore
$f^{*}p_{*}\shF\to\alpha_{*}t^{*}\shF$ is an isomorphism (see \cite[Prop 1.15]{Tonini2014}).
The second map is an isomorphism because $\alpha$ is an isomorphism.
Pulling back again along $s$ we obtain the result.

$1)\then3)$. Set $B=\Hl^{0}(\odi{\stX})$ and $p\colon\stX\arr\Spec B$
the induced map. Notice that $\L_{B}(\{\odi{\stX}\},B)\simeq\Mod B$
and under this isomorphism 
\[
\Omega^{*}\colon\QCoh\stX\arr\L_{B}(\{\odi{\stX}\},B)\text{ and }\shF_{*,\{\odi{\stX}\}}\colon\L_{B}(\{\odi{\stX}\},B)\arr\QCoh\stX
\]
 correspond to $p_{*}\colon\QCoh\stX\arr\Mod B$ and $p^{*}\colon\Mod B\arr\QCoh\stX$
respectively. Since $\{\odi{\stX}\}$ generates $\QCoh(\stX)$ by
hypothesis we can apply \ref{thm:main thm main paper} and conclude
that $p_{*}\colon\QCoh\stX\arr\Mod B$ is fully faithful, $p^{*}\colon\Mod B\to\QCoh(\stX)$
is exact, that is $p\colon\stX\to\Spec B$ is flat by \cite[Prop 1.14]{Tonini2014},
and $p^{*}p_{*}\simeq\id$.

We want to show that $\stX\arr\Spec B$ is fully faithful or, equivalently,
that the diagonal $\stX\arr\stX\times_{B}\stX$ is an equivalence.
Let $h\colon V=\Spec C\arr\stX$ be a representable fpqc covering
and consider the Cartesian diagrams      \[   \begin{tikzpicture}[xscale=2.5,yscale=-1.2]     \node (A0_0) at (0, 0) {$\stX_V$};     \node (A0_1) at (1, 0) {$\stX$};     \node (A0_2) at (2, 0) {$V$};     \node (A0_3) at (3, 0) {$\stX$};     \node (A1_0) at (0, 1) {$V$};     \node (A1_1) at (1, 1) {$\Spec B$};     \node (A1_2) at (2, 1) {$\stX\times_B V$};     \node (A1_3) at (3, 1) {$\stX\times_B \stX$};     \path (A0_0) edge [->]node [auto] {$\scriptstyle{}$} (A0_1);     \path (A0_2) edge [->]node [auto] {$\scriptstyle{h}$} (A0_3);     \path (A1_0) edge [->]node [auto] {$\scriptstyle{f}$} (A1_1);     \path (A0_3) edge [->]node [auto] {$\scriptstyle{\Delta}$} (A1_3);     \path (A0_2) edge [->]node [auto] {$\scriptstyle{s}$} (A1_2);     \path (A0_0) edge [->]node [auto] {$\scriptstyle{}$} (A1_0);     \path (A0_1) edge [->]node [auto] {$\scriptstyle{p}$} (A1_1);     \path (A1_2) edge [->]node [auto] {$\scriptstyle{\id\times h}$} (A1_3);   \end{tikzpicture}   \] where
$f=ph\colon V\to\Spec B$. We are going to prove that $\stX_{V}\to V$
is an isomorphism, but first we show how to conclude the proof from
this. Since $s$ is a section of $\stX_{V}\to V$ it would be an isomorphism
too. Since the horizontal arrow in the second diagram is an fpqc covering,
by descent it would follow that $\Delta$ is an equivalence. Moreover
if if $\stX\times_{B}k\neq\emptyset$ for all geometric points then
$f\colon V\arr\Spec B$ would be an fpqc covering. The fact that $\stX_{V}\to V$
is an equivalence would therefore imply that $p\colon\stX\to\Spec B$
is an equivalence.

Thus our goal is to show that $\stX_{V}\to V$ is an isomorphism.
Since $p$ is flat also $f$ is flat. In particular $\Hl^{0}(\odi{\stX})=B$
implies $\Hl^{0}(\odi{\stX_{V}})=C$. Since $\stX_{V}\arr\stX$ is
affine and $\{\odi{\stX}\}$ generates $\QCoh(\stX)$ it follows that
$\{\odi{\stX\times_{B}V}\}$ generates $\QCoh(\stX\times_{B}V)$ (see
\cite[Rem 7.3]{Tonini2014}). Thus $\stX_{V}$ is pseudo-affine.

Discussion above shows that we can assume that $p\colon\stX\arr\Spec B$
has a section, that we still denote by $s\colon\Spec B\arr\stX$,
and we have to show that it is an isomorphism.

We first prove that $p_{*}\colon\QCoh\stX\arr\Mod B$ is an equivalence.
It suffices to show that, if $M\in\Mod B$, then the map $\gamma_{M}\colon M\arr p_{*}p^{*}M$
is an isomorphism. Notice that $p^{*}\gamma_{M}$ is a section of
the map $(p^{*}p_{*})p^{*}M\arr p^{*}M$ which is an isomorphism.
So $p^{*}\gamma_{M}$ and $\gamma_{M}=s^{*}p^{*}\gamma_{M}$ are isomorphisms.
Since $p_{*}s_{*}\colon\Mod B\arr\Mod B$ is the identity, we can
conclude that $s_{*}\simeq p^{*}$ and $s^{*}\simeq p_{*}$. Consider
the Cartesian diagram   \[   \begin{tikzpicture}[xscale=1.5,yscale=-1.2]     \node (A0_0) at (0, 0) {$U$};     \node (A0_1) at (1, 0) {$\Spec C$};     \node (A1_0) at (0, 1) {$\Spec B$};     \node (A1_1) at (1, 1) {$\stX$};     \path (A0_0) edge [->]node [auto] {$\scriptstyle{g}$} (A0_1);     \path (A1_0) edge [->]node [auto] {$\scriptstyle{s}$} (A1_1);     \path (A0_1) edge [->]node [auto] {$\scriptstyle{h}$} (A1_1);     \path (A0_0) edge [->]node [auto] {$\scriptstyle{t}$} (A1_0);   \end{tikzpicture}   \] 
where $h$ is a representable fpqc covering. Since $\stX$ has quasi-affine
diagonal it follows that $U$ is a quasi-affine scheme. Moreover
\[
\Hl^{0}(\odi U)\simeq g_{*}t^{*}\odi B\simeq h^{*}s_{*}\odi B\simeq h^{*}p^{*}\odi B\simeq C
\]
Thus $g\colon U\arr\Spec C$ and, by descent, $s\colon\Spec B\arr\stX$
are open immersions. We finish by proving that $g$ is surjective.
Let $Z$ be the complement of $U$ in $\Spec C$ with reduced structure.
We have
\[
\Hl^{0}(\odi Z)\simeq p_{*}h_{*}\odi Z\simeq s^{*}h_{*}\odi Z=t_{*}g^{*}\odi Z=0
\]
Thus $Z$ is empty as required.
\end{proof}
\begin{rem}
\label{rem: classifying stack of elliptic curve} The assumption on
the diagonal in \ref{lem:recognizing affine rings from global sections}
is necessary: the stack $\stX=\Bi_{k}E$, where $E$ is an elliptic
curve over a field $k$, is not a sheaf but $\QCoh\stX\simeq\QCoh k$.
\end{rem}

\begin{rem}
\label{rem: generating subcategory under pseudo-affine maps} As a
consequence of \ref{lem:recognizing affine rings from global sections}
a pseudo-affine stack is separated. In particular a pseudo-affine
map $f\colon\stY\arr\stX$ of pseudo-algebraic fiber categories is
quasi-compact and separated. Taking an affine atlas of $\stX$ and
applying \ref{lem:recognizing affine rings from global sections}
we can also conclude that $f^{*}f_{*}\shF\arr\shF$ is surjective
for all $\shF\in\QCoh\stY$. In particular if $\shD\subseteq\QCoh\stX$
generates $\QCoh\stX$ then $f^{*}\shD$ generates $\QCoh\stY$ (see
\cite[Rem 7.3]{Tonini2014}).
\end{rem}

\begin{rem}
\label{rem:absolute relative pseudo-affine}If $\stY\arr\Spec A$
is a map of fibered categories then $\stY$ is pseudo-affine if and
only if $\stY\arr\Spec A$ is pseudo-affine, because if $\stY$ is
pseudo-affine and $g\colon\stY'\arr\stY$ is an affine map, then $\stY'$
is an fpqc sheaf with quasi-affine diagonal and $\odi{\stY'}=g^{*}\odi{\stY}$
generates $\QCoh\stY'$.
\end{rem}

\begin{rem}
\label{rem:fpqc descent of pseudo-affine} Consider a Cartesian diagram  \[   \begin{tikzpicture}[xscale=1.6,yscale=-1.2]     \node (A0_0) at (0, 0) {$Y$};     \node (A0_1) at (1, 0) {$\stY$};     \node (A1_0) at (0, 1) {$X$};     \node (A1_1) at (1, 1) {$\stX$};     \path (A0_0) edge [->]node [auto] {$\scriptstyle{}$} (A0_1);     \path (A1_0) edge [->]node [auto] {$\scriptstyle{}$} (A1_1);     \path (A0_1) edge [->]node [auto] {$\scriptstyle{}$} (A1_1);     \path (A0_0) edge [->]node [auto] {$\scriptstyle{}$} (A1_0);   \end{tikzpicture}   \] of
pseudo-algebraic fpqc stacks such that $X\to\stX$ is a representable
fpqc covering. If $Y\to X$ is pseudo-affine then so is $\stY\to\stX$.

By standard arguments of descent we can assume $\stX=\Spec B$ and
$X=\Spec B'$ affine. So $Y$ is pseudo-affine by hypothesis. Since
$Y\to\stY$ is affine it follows that $\stY$ is quasi-compact with
affine diagonal. In particular, since $B\arr B'$ is flat, we have
$\Hl^{0}(\odi Y)\simeq\Hl^{0}(\odi{\stY})\otimes_{B}B'$ and therefore
we can assume $\Hl^{0}(\odi{\stY})=B$ and $\Hl^{0}(\odi Y)=B'$.
In this case $Y\arr X$ is flat and fully faithful and, since $\stX$
and $\stY$ are fpqc stacks, it follows that also $\stY\arr\stX$
is flat and fully faithful.
\end{rem}

\begin{thm}
\label{thm:flat mono are quasi-affine} A quasi-compact flat monomorphism
of algebraic stacks is quasi-affine.
\end{thm}

\begin{proof}
We have to prove that if $\stX\to\Spec R$ is a flat monomorphism
and $\stX$ is quasi-compact algebraic space then $\stX$ is a quasi-affine
scheme. First we observe that $\stX$ is separated because $\stX\to\stX\times_{R}\stX$
is an isomorphism. Thus we can apply \ref{lem:recognizing affine rings from global sections}
and conclude that $\stX$ is pseudo-affine. Moreover by \cite[Tag 0B8A]{SP014}
the space $\stX$ is actually a scheme. By \cite[Prop 5.1.2]{EGAII}
we conclude that $\stX$ is quasi-affine.
\end{proof}
\begin{cor}
\label{cor:pseudo-affine algebraic is quasi-affine}A pseudo-affine
algebraic stack is a quasi-affine scheme.
\end{cor}

\begin{proof}
Use \ref{lem:recognizing affine rings from global sections} and \ref{thm:flat mono are quasi-affine}.
\end{proof}
The following property is known for algebraic stacks (see \cite[Corollary 5.11]{Gross2013}
and \cite[Proposition 1.3]{Totaro2002}).
\begin{cor}
\label{cor:quasi-affine and resolution implies affine} A quasi-compact
fpqc stack $\stX$ with quasi-affine diagonal and with the resolution
property has affine diagonal.
\end{cor}

\begin{proof}
Since $\stX$ is pseudo-algebraic the category $\Loc\stX$ is essentially
small. Thus we can consider a set $\shR$ of representatives of isomorphism
classes of locally free sheaves over $\stX$. Given $\E\in\shR$ we
define the sheaf
\[
\Fr(\E)=\bigsqcup_{n\in\N}\Isosh_{\stX}(\odi{\stX}^{n},\E)\colon(\Aff/\stX)^{\text{op}}\arr\sets\comma\Fr(\E)(U\arrdi f\stX)=\bigsqcup_{n\in\N}\Iso_{U}(\odi U^{n},f^{*}\E)
\]
The map $\Fr(\E)\arr\stX$ is an affine fpqc covering: $\stX=\sqcup_{n}\stX_{n}$
where $\stX_{n}$ is the locus where $\E$ has rank $n$ and $\Fr(\E)_{|\stX_{n}}=\Isosh_{\stX}(\odi{\stX}^{n},\E)\to\stX_{n}$
is a $\Gl_{n}$-torsor. In particular $g\colon\Fr=\prod_{\E\in\shR}\Fr(\E)\arr\stX$
is also an affine fpqc covering and $\Fr$ is quasi-compact with quasi-affine
diagonal.

Since $g$ is affine by \ref{rem: generating subcategory under pseudo-affine maps}
$g^{*}\Loc\stX$ generates $\QCoh\Fr$. On the other hand if $\E\in\Loc\stX$
then by construction $\Fr$ is a (finite) disjoint union of open substacks
over which $g^{*}\E$ is free, which implies that $g^{*}\E$ is generated
by global sections. We can conclude that $\Fr$ is a pseudo-affine
sheaf and, by \ref{rem: generating subcategory under pseudo-affine maps},
that it has affine diagonal.

In particular, taking an affine atlas of $\Fr$, we get an affine
map faithfully flat map $V\to\stX$ from an affine scheme. In particular
$V\times_{\stX}V$ is affine and $V\times_{\stX}V\to V\times V$ is
an affine map. Since this last map is the base change of the diagonal
$\stX\to\stX\times\stX$ along the fpqc covering $V\times V\to\stX\times\stX$
we can conclude that the diagonal $\stX\to\stX\times\stX$ is affine
as well.
\end{proof}
We prove the following result to compare our results with \cite{Savin2006}.
\begin{prop}
\label{rem: cases for resolution property} Let $G$ be a flat and
affine group scheme over $R$ such that $\Bi_{R}G$ has the resolution
property. Then:
\begin{enumerate}
\item If $U$ is a pseudo-affine sheaf over $R$ with an action of $G$
then $[U/G]$ has the resolution property.
\item If $\stX=[X/G]$ for a quasi-compact scheme $X$ and there exists
$\shL\in\Pic(\stX)$ whose pullback to $X$ is very ample relatively
to $R$ then $\stX$ has the resolution property.
\end{enumerate}
\end{prop}

\begin{proof}
1) By \ref{rem:absolute relative pseudo-affine} and \ref{rem:fpqc descent of pseudo-affine}
the map $[U/G]\arr\Bi_{R}G$ is pseudo-affine. By \ref{rem: generating subcategory under pseudo-affine maps}
we obtain the conclusion.

2) Consider the Cartesian diagrams   \[   \begin{tikzpicture}[xscale=2.2,yscale=-1.2]     \node (A0_0) at (0, 0) {$U$};     \node (A0_1) at (1, 0) {$\Spec R$};     \node (A1_0) at (0, 1) {$X$};     
\node (A1_1) at (1, 1) {$\Bi_R \Gm$};     
\node (A1_2) at (2, 1) {$\Spec R$};     \node (A2_0) at (0, 2) {$\stX$};     
\node (A2_1) at (1, 2) {$\Bi_R(G\times \Gm)$};     
\node (A2_2) at (2, 2) {$\Bi_R G$};     
\path (A0_1) edge [->]node [auto] {$\scriptstyle{}$} (A1_1);     \path (A0_0) edge [->]node [auto] {$\scriptstyle{}$} (A0_1);     \path (A2_0) edge [->]node [auto] {$\scriptstyle{}$} (A2_1);     \path (A1_0) edge [->]node [auto] {$\scriptstyle{}$} (A1_1);     \path (A1_1) edge [->]node [auto] {$\scriptstyle{}$} (A1_2);     \path (A1_0) edge [->]node [auto] {$\scriptstyle{}$} (A2_0);     \path (A1_1) edge [->]node [auto] {$\scriptstyle{}$} (A2_1);     \path (A0_0) edge [->]node [auto] {$\scriptstyle{}$} (A1_0);     
\path (A2_1) edge [->]node [auto] {$\scriptstyle{}$} (A2_2);     
\path (A1_2) edge [->]node [auto] {$\scriptstyle{}$} (A2_2);   \end{tikzpicture}   \]  It follows that $\stX\simeq[U/G\times\Gm]$. Now consider the Cartesian
diagrams   \[   \begin{tikzpicture}[xscale=2.2,yscale=-1.2]     \node (A0_0) at (0, 0) {$U$};     \node (A0_1) at (1, 0) {$\A^n_R - \{0\}$};     \node (A0_2) at (2, 0) {$\Spec R$};     \node (A1_0) at (0, 1) {$X$};     \node (A1_1) at (1, 1) {$\PP^n_R$};     \node (A1_2) at (2, 1) {$\Bi_R \Gm$};     \path (A0_0) edge [->]node [auto] {$\scriptstyle{}$} (A0_1);     \path (A0_1) edge [->]node [auto] {$\scriptstyle{}$} (A1_1);     \path (A1_0) edge [->]node [auto] {$\scriptstyle{i}$} (A1_1);     \path (A1_1) edge [->]node [auto] {$\scriptstyle{}$} (A1_2);     \path (A0_2) edge [->]node [auto] {$\scriptstyle{}$} (A1_2);     \path (A0_0) edge [->]node [auto] {$\scriptstyle{}$} (A1_0);     \path (A0_1) edge [->]node [auto] {$\scriptstyle{}$} (A0_2);   \end{tikzpicture}   \] where
$i\colon X\to\PP_{R}^{n}$ is the immersion induced by $\shL$. Since
$X$ is quasi-compact and $\PP_{R}^{n}$ is separated it follows that
$U$ is quasi-compact. Moreover $U\to\A_{R}^{n}$ is an immersion
and, since $\A_{R}^{n}$ is affine, it is quasi-compact. From \cite[Tag 01QV]{SP014}
we see that $U$ is a quasi-compact open of a closed subscheme of
$\A_{R}^{n}$ and therefore it is quasi-affine.

Thus by 1) we just have to show that $\Bi_{R}(G\times\Gm)$ has the
resolution property. Let $a\colon\Bi_{R}(G\times\Gm)=\Bi_{R}G\times\Bi_{R}\Gm\to\Bi_{R}G$
be the projection and $\shN$ the pullback to $\Bi_{R}(G\times\Gm$)
of the canonical invertible sheaf of $\Bi_{R}\Gm$. Given $\shF\in\QCoh(\Bi_{R}(G\times\Gm)$)
there is a canonical map 
\[
\bigoplus_{n\in\Z}a^{*}a_{*}(\shF\otimes\shN^{-\otimes n})\otimes\shN^{\otimes n}\to\shF
\]
Going fpqc locally on $\Bi_{R}G$ and using the usual representation
theory of $\Gm$ we see that the above map is an isomorphism. It is
therefore clear that the sheaves of the form $a^{*}\shG\otimes\shN^{\otimes n}$
for $\shG\in\Loc(\Bi_{R}G)$ generates $\QCoh(\Bi_{R}(G\times\Gm$).
\end{proof}

\section{Tannaka Reconstruction}

The goal of this section is to introduce the stack of fiber functors
and prove Theorem \ref{thma}. We will work over a base ring $R$
and denote by $A$ a general $R$-algebra.
\begin{defn}
\label{def:definition of Fib stacks} Let $\stX$ be a quasi-compact
fiber category over $R$, $\shC\subseteq\Loc\stX$ a monoidal subcategory
and $A$ an $R$-algebra. We define $\SMex_{R}(\shC,A)$ as the category
of contravariant, $R$-linear and \emph{strong} monoidal functors
$\Gamma\colon\shC\to\Mod A$ such that, for all geometric points $\Spec k\arr\Spec A$,
$\Gamma\otimes_{A}k$ are left exact in the sense of \ref{def:basic categories}.

Let $\stY$ be another fibered category over $R$. A functor $\shC\arr\Loc\stY$
is said a \emph{fiber }functor if it is a \emph{covariant}, $R$-linear
and strong monoidal functor which is right exact in the sense of \ref{def:exactness}.
We denote by $\Fib_{\stX,\shC}(\stY)$ the category of fiber functors
$\shC\to\Vect(\stY)$.

We define $\Fib_{\stX,\shC}$ as the fiber category (not necessarily
in groupoids) over $R$ whose fiber over an $R$-algebra $A$ is $\Fib_{\stX,\shC}(\Spec A)$
and we call $\shP_{\shC}$ the functor
\begin{equation}
\shP_{\shC}\colon\stX\arr\Fib_{\stX,\shC}\comma(\Spec A\arrdi s\stX)\longmapsto(s^{*}\colon\shC\arr\Loc A)\label{eq:tannaka reconstruction functor}
\end{equation}
The fibered category $\Fib_{\stX,\shC}$ is called the \emph{stack
of fiber functors.}
\end{defn}

We will prove the following result and, in particular, Theorem \ref{thma}.
\begin{thm}
\label{thm: Tannaka reconstrion for Loc X} Let $\stX$ be a quasi-compact
fpqc stack over $R$ with quasi-affine diagonal, $A$ be an $R$-algebra
and $\shC\subseteq\Loc\stX$ be a monoidal subcategory with duals
that generates $\QCoh\stX$. Then the functors   \[   \begin{tikzpicture}[xscale=4.2,yscale=-0.6]     \node (A0_0) at (0, 0) {$(\Spec \alA_{\Gamma,\shC }\arr \stX)$};     \node (A0_1) at (1, 0) {$(\Gamma \colon \shC\arr\Mod A)$};     \node (A1_0) at (0, 1) {$\stX(A)$};     \node (A1_1) at (1, 1) {$\SMex_R(\shC,A)$};     \node (A2_0) at (0, 2) {$(s\colon \Spec A\arr\stX)$};     \node (A2_1) at (1, 2) {$(\duale {(s^*_{|\shC})}\colon \shC\arr\Loc A)$};     \path (A1_0) edge [->]node [auto] {$\scriptstyle{}$} (A1_1);     \path (A2_0) edge [|->,gray]node [auto] {$\scriptstyle{}$} (A2_1);     \path (A0_1) edge [|->,gray]node [auto] {$\scriptstyle{}$} (A0_0);   \end{tikzpicture}   \] 
are well defined and quasi-inverses of each other. In particular the
functor $\shP_{\shC}\colon\stX\arr\Fib_{\stX,\shC}$ is an equivalence
of stacks.
\end{thm}

An immediate corollary and generalization of Theorem \ref{thm: Tannaka reconstrion for Loc X}
is the following.
\begin{cor}
\label{cor:Lurie s point of view} Let $\stX$ be a quasi-compact
fpqc stack over $R$ with quasi-affine diagonal, $\shC\subseteq\Loc\stX$
be a monoidal subcategory with duals that generates $\QCoh\stX$ and
$\stY$ be a fibered category over $R$. Then the functor 
\[
\Hom(\stY,\stX)\arr\Fib_{\stX,\shC}(\stY)\comma(\stY\arrdi f\stX)\longmapsto f_{|\shC}^{*}\colon\shC\arr\Loc(\stY)
\]
is an equivalence of categories.
\end{cor}

\begin{proof}
The map in the statement is obtained applying $\Hom(\stY,-)$ to the
functor $\shP_{\shC}\colon\stX\arr\Fib_{\stX,\shC}$, which is an
equivalence by \ref{thm: Tannaka reconstrion for Loc X}.
\end{proof}
\begin{rem}
\label{rem:comparison with classical tannakas reconstruction} Let
$R=k$ be a field, $\stX=\Delta$ be an affine gerbe over $k$ and
$\shC=\Vect(\Delta)$, which is a $k$-Tannakian category. The stack
of fiber functor $\Pi_{\shC}\to\Aff/k$ usually associated with $\shC$
is defined as 
\[
\Pi_{\shC}(A)=\{k\text{-linear, exact and strong monoidal functors }\shC\to\Vect(A)\}
\]
Sometimes it is also required that those functors are faithful. In
any case, by \cite[Cor 2.10]{CT90} all those notions coincide and
$\Pi_{\shC}=\Fib_{\Delta,\shC}$. Classical Tannaka's reconstruction
states that the functor $\shP_{\shC}\colon\Delta\to\Pi_{\shC}=\Fib_{\Delta,\shC}$
is an equivalence. On the other hand $\shC=\Vect(\Delta)$ generates
$\QCoh(\Delta)$ (see \cite[Corollary 3.9]{CT90}), so that Theorem
\ref{thm: Tannaka reconstrion for Loc X} can be seen of a generalization
of classical Tannaka's reconstruction.
\end{rem}

\begin{rem}
Even though we keep choosing a base ring $R$, the fibered category
$\Fib_{\stX,\shC}\to\Aff/R\to\Aff$ does not depend on this choice:
we may have chosen $R=\Z$ or $R=\Hl^{0}(\odi{\stX})$. Indeed if
$A$ is an $R$-algebra and $\Gamma\colon\shC\to\Mod(A)$ is an $R$-linear
functor, then $\Hl^{0}(\odi{\stX})=\End(\odi{\stX})\to\End(A)=A$
makes $A$ into a $\Hl^{0}(\odi{\stX})$-algebra and $\Gamma$ into
an $\Hl^{0}(\odi{\stX})$ linear functor.
\end{rem}

\begin{lem}
\label{lem:strong monoidal implies surjective} Let $\stX$ be a quasi-compact
fibered category over $R$, $\shC\subseteq\Loc\stX$ be a monoidal
subcategory with duals, $\Gamma\colon\shC\to\Vect(A)$ be a contravariant,
$R$-linear and strong monoidal functor and set $\alA=\alA_{\Gamma,\shC}$,
$f\colon\stY=\Spec\alA\to\stX$. Then
\[
\Omega_{\E}^{\alA}\simeq\Hl^{0}((f^{*}\E)^{\vee})
\]
and the map
\[
\Gamma_{\E}\otimes_{A}\odi{\stY}\to(f^{*}\E)^{\vee}
\]
induced by $\Gamma\to\Omega^{\alA}$ is surjective.
\end{lem}

\begin{proof}
Set $\pi\colon\stX_{A}\to\stX$ for the projection. We have monoidal
isomorphisms
\[
\Omega_{\E}^{\alA}=\Hom(\pi^{*}\E,\alA)\simeq\Hom(\E,f_{*}\odi{\stY})\simeq\Hom(f^{*}\E,\odi{\stY})=\Hl^{0}(\duale{(f^{*}\E)})
\]
In particular we obtain a monoidal natural transformation $\Gamma\to\Omega^{\alA}=\Hl^{0}((f^{*}-)^{\vee})$
and therefore a commutative diagram    \[   \begin{tikzpicture}[xscale=3.0,yscale=-1.2]     \node (A0_0) at (0, 0) {$\Gamma_\E \otimes \Gamma_{\E^\vee}$};     \node (A0_1) at (1, 0) {$\Hl^0(f^*\E^\vee) \otimes \Hl^0(f^*\E) $};     \node (A1_0) at (0, 1) {$\Gamma_{\E\otimes \E^\vee}$};     \node (A1_1) at (1, 1) {$\Hl^0(f^*(\E\otimes \E^\vee)^\vee)$};     \path (A0_0) edge [->]node [auto] {$\scriptstyle{}$} (A0_1);     \path (A0_0) edge [->]node [auto] {$\scriptstyle{}$} (A1_0);     \path (A0_1) edge [->]node [auto] {$\scriptstyle{\omega}$} (A1_1);     \path (A1_0) edge [->]node [auto] {$\scriptstyle{}$} (A1_1);   \end{tikzpicture}   \] 
Consider the evaluation $e\colon\E\otimes\E^{\vee}\to\odi{\stX}$,
which is a map in $\shC$ because $\shC$ has duals. The morphism
$\Hl^{0}(f^{*}\odi{\stX}^{\vee})\to\Hl^{0}(f^{*}(\E\otimes\E^{\vee})^{\vee})$
maps $1$ to an element that we denote by $\psi$. After the usual
identifications $\E\otimes\E^{\vee}\simeq\Endsh(\E)$, the map $\omega$
become the evaluation
\[
\omega\colon\Hl^{0}(f^{*}\E)\otimes\Hom_{\stY}(f^{*}\E,\odi{\stY})\arr\End_{\stY}(f^{*}\E)\comma\omega(x\otimes\phi)(y)=x\phi(y)
\]
while $\psi$ become $\id_{f^{*}\E}$.

By hypothesis the vertical map on the left in the above diagram is
an isomorphism. By functoriality, there exist $x_{1},\dots,x_{n}\in\Hl^{0}(f^{*}\E)$,
$\phi_{1},\dots,\phi_{n}\in\Hom(f^{*}\E,\odi{\stX})$, coming respectively
from $\Gamma_{\E^{\vee}}$ and $\Gamma_{\E}$, such that $\id_{f^{*}\E}=\omega(\sum_{i}x_{i}\otimes\phi_{i})$.
This implies that the map $\odi{\stY}^{n}\arr f^{*}\E$ given by the
global sections $x_{1},\dots,x_{n}$ is surjective. Since this map
factors through $\Gamma_{\E^{\vee}}\otimes_{A}\odi{\stY}\to f^{*}\E$
by construction, this ends the proof.
\end{proof}
\begin{lem}
\label{lem:Fib and smex the same} Let $\stX$ be a quasi-compact
fibered category over $R$ and $\shC\subseteq\Loc\stX$ be a monoidal
subcategory. If $\Gamma\in\SMex_{R}(\shC,A)$ and $\E'\xrightarrow{\beta}\E\to0$
or $\E''\arrdi{\alpha}\E'\arrdi{\beta}\E\to0$ is a finite test sequence
in $\shC$ then 
\[
0\to\Gamma_{\E}\to\Gamma_{\E'}\to\Coker(\Gamma_{\beta})\to0\text{ and }0\to\Coker(\Gamma_{\beta})\to\Gamma_{\E''}\to\Coker(\Gamma_{\alpha})\to0
\]
are short exact sequences of vector bundles over $A$.

In particular $\Gamma$, as well as any base change $\Gamma\otimes_{A}B$
for an $A$-algebra $B$, are left exact. Moreover $\Gamma\otimes_{A}B\in\SMex_{R}(\shC,B)$.
\end{lem}

\begin{proof}
By \cite[Tag 046Y]{SP014} if $M\to N$ is a map of vector bundles
over $A$ and, for all $p\in\Spec A$, the map $M\otimes_{A}k(p)\to N\otimes_{A}k(p)$
is injective then $M\to N$ is injective and $N/M$ is flat. As $N/M$
is finitely presented we can moreover conclude that $N/M$ is a vector
bundle. Applying this on the map $\Gamma_{\E}\to\Gamma_{\E'}$ we
obtain the first exact sequence and that $\Coker(\Gamma_{\beta})$
is locally free. Here we use the definition of $\SMex$, which also
tell us that $\gamma\colon\Coker(\Gamma_{\beta})\to\Gamma_{\E''}$
is injective on the geometric point. Again we can conclude that this
map is injective and that $\Coker(\gamma)$ is a vector bundle. As
$\Coker(\gamma)=\Coker(\Gamma_{\alpha})$ this concludes the proof.
\end{proof}
By definition $\SMex_{R}(\shC,-)$ is a stack (not necessarily in
groupoids) over $\Aff/R$. The next result tells us that this stack
is just $\Fib_{\stX,\shC}$.
\begin{prop}
\label{lem:fib and smex then same stack} Let $\stX$ be a quasi-compact
fibered category over $R$ and $\shC\subseteq\Loc\stX$ be a monoidal
subcategory. Then 
\[
\Fib_{\stX,\shC}\to\SMex_{R}(\shC,-)\comma\Gamma\longmapsto\Gamma^{\vee}
\]
is an equivalence of stacks over $\Aff/R$.
\end{prop}

\begin{proof}
If $\Gamma\colon\shC\to\Vect(A)$ is a covariant (resp. contravariant),
$R$-linear and strong monoidal functor then $\Gamma^{\vee}\colon\shC\to\Vect(A)$
is a contraviariant (resp. covariant) $R$-linear functor which moreover
has a strong monoidal structure. Moreover if $B$ is an $A$-algebra
we have 
\[
\Hom_{A}(\Gamma_{\E},A)\otimes_{A}B\simeq\Hom_{B}(\Gamma_{\E}\otimes_{A}B,B)
\]
because $\Gamma_{\E}\in\Vect(A)$.

If $\Gamma\in\SMex_{R}(\shC,A)$ then $\Gamma^{\vee}\in\Fib_{\stX,\shC}(A)$
thanks to \ref{lem:Fib and smex the same}. Conversely assume $\Gamma\in\Fib_{\stX,\shC}(A)$.
We must show that $(\Gamma\otimes_{A}k)^{\vee}\simeq\Gamma^{\vee}\otimes_{A}k$
is left exact for all geometric points $\Spec k\to\Spec A$. As $\Gamma\otimes_{A}k\in\Fib_{\stX,\shC}(k)$
we can assume that $A$ is a field. The functor $\Gamma^{\vee}\colon\shC\to\Vect(A)$
is left exact because $\Gamma$ is exact on all finite test sequences
and the dual of a right exact sequence is again exact.
\end{proof}
\begin{proof}
(of Theorem \ref{thm: Tannaka reconstrion for Loc X}). The last claim
follows from \ref{lem:fib and smex then same stack}. Set $\pi\colon\stX_{A}\to\stX$
for the projection. Composing by $\pi$ we obtain an equivalence $\stX_{A}(A)\to\stX(A)$.
Since $\stX$ and therefore $\stX_{A}$ have affine diagonal by \ref{cor:quasi-affine and resolution implies affine},
all morphisms $\Spec A\arr\stX,\stX_{A}$ are affine. Therefore the
functor $\stX_{A}(A)\arr\QAlg(\stX_{A})$ which maps $t\colon\Spec A\arr\stX_{A}$
to $t_{*}\odi A$ is fully faithful. By \ref{thm:main thm main paper}
and the fact that 
\[
\Omega_{\E}^{t_{*}\odi A}=\Hom_{\stX_{A}}(\pi^{*}\E,t_{*}\odi A)\simeq\Hom_{\stX}(\E,s_{*}\odi A)\simeq\duale{(s^{*}\E)}\text{ for }t\in\stX_{A}(A)\comma\E\in\shC
\]
where $s=\pi t\colon\Spec A\to\stX$ we can conclude that the functor
$\stX(A)\arr\SMex_{R}(\shC,A)$, $s\longmapsto\duale{(s_{|\shC}^{*})}$
is well defined and fully faithful.

Set $\alA=\alA_{\Gamma,\shC}\in\QAlg(\stX_{A})$. We must show that,
given $\Gamma\in\SMex_{R}(\shC,A)$, the composition $p\colon\Spec\alA\arr\stX_{A}\arr\Spec A$
is an isomorphism. Set $\stY=\Spec\alA$ and $f\colon\stY\arr\stX$
for the structure morphism. We want to apply \ref{lem:recognizing affine rings from global sections}
on $p\colon\stY\arr\Spec A$.

Notice that, since $f\colon\stY\to\stX$ is affine, $\stY$ is quasi-compact
and has affine diagonal. Moreover by \ref{thm:main thm main paper}
and \ref{lem:strong monoidal implies surjective} we have monoidal
isomorphisms
\[
\Gamma_{\E}\simeq\Omega_{\E}^{\alA}\simeq\Hl^{0}(\duale{(f^{*}\E)})
\]
In particular, since $\Gamma$ is a strong monoidal functor, the isomorphism
$A\to\Gamma_{\odi{\stX}}$ yields the isomorphism $A\arr\Hl^{0}(\alA)$.

Let $\Spec k\arr\Spec A$ be a geometric point and set $g\colon\stX_{k}\arr\stX_{A}$
for the base change map. We have $\stY\times_{A}k\simeq\Spec(g^{*}\alA)$,
while by \cite[Prop 2.14]{Tonini2014} we have $g^{*}\alA\simeq\alA_{\Gamma\otimes_{A}k,\shC}$.
Since $\Gamma\otimes_{A}k$ is also left exact by hypothesis, by \ref{thm:main thm main paper}
we get $\Gamma\otimes_{A}k\simeq\Omega^{\alA_{\Gamma\otimes_{A}k,\shC}}$.
Thus $\Gamma_{\odi{\stX}}\otimes_{A}k=A\otimes_{A}k\simeq k$ implies
that $\Omega_{\odi{\stX}}^{\alA_{\Gamma\otimes_{A}k,\shC}}\simeq k$
and therefore $\alA_{\Gamma\otimes_{A}k,\shC}\neq0$, that is $\stY\times_{A}k\neq\emptyset$.

It remains to show that $\{\odi{\stY}\}$ generates $\QCoh\stY$.
Since $\stY\arrdi f\stX$ is affine, $f^{*}\shC$ generates $\QCoh\stY$
by \ref{rem: generating subcategory under pseudo-affine maps}. On
the other hand every sheaf $f^{*}\E$ is generated by global sections
thanks to \ref{lem:strong monoidal implies surjective}. In conclusion
$\stY$ is pseudo-affine and, by \ref{lem:recognizing affine rings from global sections},
the map $p\colon\stY\to\Spec A$ is an isomorphism.
\end{proof}

\section{Pseudo-Affine sheaves revisited}

In this section we give an alternative characterization of pseudo-affine
sheaves, which will be used when studying the fiber category $\Fib_{\stX,\shC}$.
\begin{lem}
\label{lem:locus of exactness} Let $\pi\colon\stX\to\Aff/R$ be a
fibered category over $R$, $\shT_{*}$ be a bounded above complex
of locally free sheaves on $\stX$ and denote by $\stU_{l}$, for
$l\in\Z$, the locus in $\stX$ where $\shT_{*}$ is exact in degrees
greater than $l$, that is
\[
\stU_{l}(A)=\{\xi\colon\Spec A\to\stX\st\xi^{*}\shT_{*}\text{ is exact in degrees greater than }l\}\subseteq\stX(A)
\]
Then $\stU_{l}\to\stX$ is a quasi-compact open immersion.
\end{lem}

\begin{proof}
Since $\shT_{*}$ is bounded above we have $\shU_{l}=\stX$ for $l\gg0$.
It is therefore enough to show that $\stU_{l-1}\to\stU_{l}$ is a
quasi-compact open immersion. In particular we can assume $\stX$
affine and $\stU_{l}=\stX$, that is assume $\shT_{*}$ exact in degrees
greater than $l$. Set $\E=\Ker(\shT_{l}\to\shT_{l+1})$ and consider
the complex 
\[
0\to\E\to\shT_{l}\to\shT_{l+1}\to\cdots
\]
By construction this complex is exact. Since all the $\shT_{l}$ are
locally free it follows that $\E$ is locally free as well and that
this complex remains exact after any pullback. We can therefore conclude
that $\stU_{l-1}$ is the locus where the map
\[
\shT_{l-1}\to\Ker(\shT_{l}\to\shT_{l+1})=\E
\]
 is surjective. In particular $\stU_{l-1}$ is a subfibered category
of $\stX$. Going Zariski locally we can assume $\shT_{l-1}$ and
$\E$ free of rank $m$ and $n$. In this case the locus $\stU_{l}$
is the locus where
\[
\Lambda^{n}\shT_{l-1}\to\Lambda^{n}\E=\det\E\simeq B
\]
is surjective. But this is the complement of the zero locus defined
by the above matrix, which is a quasi-compact open subset of $\Spec B$.
\end{proof}
\begin{thm}
\label{prop:arbitrary intersection of quasi-compact open is pseudo-affine}
An intersection of quasi-compact open subschemes (thought of as sheaves)
of an affine scheme is pseudo-affine. Conversely if $U$ is a pseudo-affine
stack then it is (equivalent to) a sheaf and it is the intersection
of the quasi-compact open subschemes of $\Spec\Hl^{0}(\odi U)$ containing
it.
\end{thm}

\begin{proof}
Let $X=\Spec B$ be an affine scheme, $\{U_{i}\}_{i\in I}$ be a set
of quasi-compact open subsets of $X$ and set $U=\cap_{i}U_{i}$.
If $i\in I$ the subscheme $U_{i}$ is the complement of the zero
locus of finitely many elements of $B$ and thus there exists a free
$B$-module $E_{i}$ and a map $\phi_{i}\colon E_{i}\arr B$ such
that $U_{i}$ is the locus where $\phi_{i}$ is surjective. Let $V_{i}\colon\Aff/B\to\sets$
be the functor
\[
V_{i}(A)=\{s\in E_{i}\otimes_{B}A\st\phi_{i}(s)=1\}
\]
It is easy to check that $V_{i}$ is affine, the map $V_{i}\to\Spec B$
factors through $U_{i}$ and that $V_{i}\to U_{i}$ is locally $(\Ker\phi_{i})_{|U_{i}}\to U_{i}$.
In particular $V_{i}\to\Spec B$ is flat. Now set $V=\prod_{i}V_{i}\colon\Aff/B\arr\sets$,
that is
\[
V(A)=\{(s_{i})_{i\in I}\st s_{i}\in E_{i}\otimes_{B}A\text{ and }\phi_{i}(s_{i})=1\}
\]
The scheme $V$ is affine, the map $V\arr\Spec B$ is flat, factors
through $U$ and $V\arr U$ is surjective (as functors). Moreover
if $\Spec A\arr U$ is any map then $V\times_{U}A=V\times_{X}A$ over
$A$ because $U\arr X=\Spec B$ is a monomorphism. It follows that
$V\to U$ is affine and faithfully flat, in particular an fpqc covering.
Thus $U$ is quasi-compact and $U\arr X$ is a flat monomorphism.
The result then follows from \ref{lem:recognizing affine rings from global sections}.

Now assume that $U$ is a pseudo-affine stack. By \ref{lem:recognizing affine rings from global sections}
we know that $p\colon U\to\Spec B$, $B=\Hl^{0}(\odi U)$, is a flat
monomorphism. Denote by $Z$ the intersection of all quasi-compact
open subsets of $\Spec B$ containing $U$ and set $\shC=\{\odi U\}$.
In particular $U\subseteq Z$.

Given a map $\alpha\colon\Spec A\arr\Spec B$ factoring through $Z$
we have to show that is factors through $U$. Consider $\alpha_{|\shC}^{*}\colon\shC\to\Loc(A)$,
which is a covariant strong monoidal functor. We show that $\alpha_{|\shC}^{*}\in\Fib_{U,\shC}(A)$.
By definition, if $\shT_{*}\colon\odi U^{m}\to\odi U^{q}\to0$ or
$\shT_{*}\colon\odi U^{n}\arr\odi U^{m}\arr\odi U^{q}\arr0$ is an
exact sequence on $U$, we must show that $\alpha_{|\shC}^{*}$ is
exact on $\shT_{*}$. The sequence $\shT_{*}$ defines a complex $\shW_{*}$
of free $A$-modules, namely $\shW_{*}=\Hl^{0}(\shT_{*})$, and the
locus $W$ in $\Spec A$ where $\shW_{*}$ is exact is quasi-compact,
open and contains $U$ (see \ref{lem:locus of exactness}). Thus $Z\subseteq W$,
the sequence $\shW_{*}$ become exact on $Z$ and therefore $\alpha^{*}$
maintains its exactness, as required.

Since $U$ is pseudo-affine and $\alpha_{|\shC}^{*}\in\Fib_{U,\shC}(A)$,
by \ref{thm: Tannaka reconstrion for Loc X} there exists $s\colon\Spec A\to U$
and an isomorphism $s_{|\shC}^{*}\simeq\alpha_{|\shC}^{*}$ in $\Fib_{U,\shC}(A)$.
On the other hand, since $B=\Hl^{0}(\odi U)$, the restriction $\{\odi{\Spec B}\}\to\shC=\{\odi U\}$
is an equivalence and therefore $\Fib_{U,\shC}(A)\to\Fib_{\Spec B,\{\odi{\Spec B}\}}$
is fully faithful. By \ref{thm: Tannaka reconstrion for Loc X} applied
on $\Spec B$ we can conclude that $\alpha\colon\Spec A\to\Spec B$
and $\Spec A\arrdi sU\to\Spec B$ coincides as required.
\end{proof}
\begin{example}
\label{exa:counterexample pseudo-affine} We show an example of a
pseudo-affine sheaf which is not quasi-affine. Let $B$ be a noetherian
normal domain with $\dim B\geq2$ and infinitely many primes $p$
with $\alt p=2$. For instance $B=k[x,y]$ for a field $k$. Let $I=\{p\in\Spec B\st\alt p\geq2\}$
and $U_{p}=\Spec B-V(p)$ for $p\in I$. Set $U=\bigcap_{p\in I}U_{p}$.
The sheaf $U$ is pseudo-affine by \ref{prop:arbitrary intersection of quasi-compact open is pseudo-affine}
and we are going to show that it is not quasi-affine. Assume by contradiction
that this is true. We identify the points of the topological space
of $U$ with the ones of its image in $\Spec B$. Remember that a
map $\Spec C\arr\Spec B$ factors through $U$ if and only if $pC=C$
for all $p\in I$. Given $q\in\Spec B$ and $k(q)\subseteq L$ an
extension of fields, we have that $\Spec L\arr\Spec B$ factors through
$U$ if and only if
\[
pL=L\:\forall p\in I\iff pk(q)=k(q)\:\forall p\in I\iff pB_{q}=B_{q}\:\forall p\in I\iff\alt q\leq1
\]
This tells us that $U=\{q\in\Spec B\st\alt q\leq1\}$ and that, for
all $q\in U$, $\Spec B_{q}$ factors through $U$. So, if $q\in U$,
The map $U\times_{B}B_{q}\arr\Spec B_{q}$ is a monomorphism with
a section and therefore an isomorphism. In particular $\odi{U,q}=B_{q}$.
If $V=\Spec A\subseteq U$ is an open subset and $q\in\Spec A$ is
a minimal prime, then $A_{q}=B_{q}$ has dimension $0$ and therefore
$q$ must be the generic point of $\Spec B$. Thus $U$ is irreducible
and, since the local rings of $U$ are all domains, integral. In particular
\[
\Hl^{0}(\odi U)=\bigcap_{q\in U}\odi{U,q}=\bigcap_{q\in\Spec B\st\alt q\leq1}B_{q}=B
\]
where the last equality follows from the fact that $B$ is normal.
Since $U$ is quasi-affine, it follows that it is an open subset of
$\Spec B$. So $Z=\Spec B-U$ has a finite number of generic points
and all $p\in\Spec B$ with $\alt p=2$ have to be generic points,
contradicting our assumptions.
\end{example}

\section{The stack of fiber functors}

The last statement of Theorem \ref{thm: Tannaka reconstrion for Loc X}
admits an almost converse. Let $\stX$ be a fibered category over
$R$ and $\shC\subseteq\Loc\stX$ be a full monoidal subcategory.
We define 
\[
\shG_{*}\colon\shC\to\Hom(\Fib_{\stX,\shC},\Loc_{R})
\]
where $\Loc_{R}$ is the stack of locally free sheaves over $R$,
mapping a sheaf $\E\in\shC$ to
\[
\shG_{\E}(\Gamma\in\Fib_{\stX,\shC}(A))=\Gamma_{\E}\in\Loc(A)\text{ for all }R\text{-algebras }A
\]
In particular the composition $\stX\to\Fib_{\stX,\shC}\arrdi{\shG_{\E}}\Loc_{R}$
is just $\E\in\shC$. Notice that, a priori, $\Fib_{\stX,\shC}$ is
not necessarily fibered in groupoids and therefore the notion of a
locally free sheaf on it is not defined (although one can easily guess
the definition). In the next result we will see that, under suitable
conditions on $\shC$, the stack $\Fib_{\stX,\shC}$ is fibered in
groupoid. In this case $\shG_{*}$ will just be a map $\shG_{*}\colon\shC\to\Loc(\Fib_{\stX,\shC})$,
which is easily seen to be a covariant strong monoidal functor.

Given a function $f\colon\shC\to\N$ we define $\Fib_{\stX,\shC}^{f}$
as the sub-fibered category of $\Fib_{\stX,\shC}$ of functors $\Gamma$
such that $\rk\Gamma_{\E}=f(\E)$ for all $\E\in\shC$.

We summarize all the main results of this section in the following:
\begin{thm}
\label{thm:when FibX,C has the resolution property} Let $\stX$ be
a quasi-compact fibered category over $R$ and $\shC\subseteq\Loc\stX$
be a full monoidal subcategory with duals. If $R$ is not a $\Q$-algebra
assume moreover that $\Sym^{n}\E\in\shC$ for all $\E\in\shC$ and
$n\in\N$. Set also 
\[
\shI=\{f\colon\shC\to\N\st\exists\xi\colon\Spec L\to\stX\text{ with }f(\E)=\rk(\xi^{*}\E)\text{ for all }\E\in\shC\}
\]
Then $\Fib_{\stX,\shC}$ is an fpqc stack in groupoids and:
\begin{enumerate}
\item $\Fib_{\stX,\shC}^{f}\neq\emptyset$ if and only if $f\in\shI$.
\item Using notation from Section \ref{subsec:Generalized-frame-bundles},
we have $D(\E)=D(\shG_{\E})$ and $Q(\E)=Q(\shG_{\E})$ for $\E\in\shC$,
so that $D(\shC)=D(\shC\to\Vect(\stX))=D(\shG_{*}\colon\shC\to\Vect(\Fib_{\stX,\shC}))$
is an affine scheme. If we think $\GL(Q(\E))$ as a group scheme over
$D(\shC)$ and set $G=\prod_{\E}\GL(\E)\to D(\shC)$ then the map
(see \ref{lem:frames})
\[
\Fib_{\stX,\shC}\to\Bi G
\]
is pseudo-affine. In particular $U=\Fr(\shG_{*}\colon\shC\to\Vect(\Fib_{\stX,\shC}))$
is pseudo-affine and $\Fib_{\stX,\shC}\simeq[U/G].$
\item $\Fib_{\stX,\shC}$ is a quasi-compact fpqc stack with affine diagonal
and the subcategory $\{\shG_{\E}\}_{\E\in\shC}\subseteq\Loc(\Fib_{\stX,\shC})$
generates $\QCoh(\Fib_{\stX,\shC})$. In particular $\Fib_{\stX,\shC}$
has the resolution property.
\item If $f\in\shI$ then the map $\Spec R\to D(\shC)$ induced by the constant
function $(f(\E))_{\E}\colon\Spec R\to\prod_{\E}\text{ranks}(\E)$
(see \ref{rem:mapping to D(E*)}) is well defined and we have Cartesian
diagrams    \[   \begin{tikzpicture}[xscale=2.7,yscale=-1.2]     \node (A0_0) at (0, 0) {$\Fib_{\stX,\shC}^f$};     \node (A0_1) at (1, 0) {$\Bi \GL_f$};     \node (A0_2) at (2, 0) {$\Spec R$};     \node (A1_0) at (0, 1) {$\Fib_{\stX,\shC}$};     \node (A1_1) at (1, 1) {$\Bi G$};     \node (A1_2) at (2, 1) {$D(\shC)$};     \path (A0_0) edge [->]node [auto] {$\scriptstyle{\omega}$} (A0_1);     \path (A0_1) edge [->]node [auto] {$\scriptstyle{}$} (A1_1);     \path (A1_0) edge [->]node [auto] {$\scriptstyle{}$} (A1_1);     \path (A0_2) edge [->]node [auto] {$\scriptstyle{}$} (A1_2);     \path (A1_1) edge [->]node [auto] {$\scriptstyle{}$} (A1_2);     \path (A0_0) edge [->]node [auto] {$\scriptstyle{}$} (A1_0);     \path (A0_1) edge [->]node [auto] {$\scriptstyle{}$} (A0_2);   \end{tikzpicture}   \] 
where $\GL_{f}=\prod_{\E}\GL_{f(\E)}$ and $\omega$ is induced by
the $(\shG_{\E})_{\E\in\shC}$. The vertical maps are flat closed
immersion, $\omega\colon\Fib_{\stX,\shC}^{f}\to\Bi\GL_{f}$ is pseudo-affine,
$U=\Fr(\shG_{*}\colon\shC\to\Vect(\Fib_{\stX,\shC}^{f}))$ is pseudo-affine,
$\Fib_{\stX,\shC}^{f}\simeq[U/\GL_{f}]$ and $\{\shG_{\E}\}_{\E\in\shC}\subseteq\Loc(\Fib_{\stX,\shC}^{f})$
generates $\QCoh(\Fib_{\stX,\shC}^{f})$. In particular $\Fib_{\stX,\shC}^{f}$
is a quasi-compact fpqc stack in groupoids with affine diagonal and
the resolution property.
\item Assume the category $\shC$ has the following two properties: there
exists a finite set $J$ of objects of $\shC$ such that any objects
of $\shC$ can be obtained by sheaves of $J$ using the operations
$\otimes$, $\oplus$, $-^{\vee},$ $\Sym^{m}$, $\Lambda^{m}$ several
times; there is a finite set of finite test sequences for $\shC$
such that if $\Gamma\colon\shC\to\Vect(A)$ is an $R$-linear strong
monoidal functor with then $\Gamma$ is exact. Then $\Fib_{\stX,\shC}$
is an algebraic stack and a quotient $[U'/\GL_{m}]$ for some quasi-affine
scheme $U'$.
\end{enumerate}
\end{thm}

\begin{rem}
\label{rem:bad algebraic fib} It is not clear if $\Fib_{\stX,\shC}^{f}$
is algebraic even if $\stX$ is a projective scheme and we choose
a $\shC$ that does not generate $\QCoh(\stX)$. The statement of
\ref{thm:when FibX,C has the resolution property}, 5) is not really
satisfactory for two reasons. The first condition requires $\shC$
to be not too big, for example $\shC=\Loc(\stX)$ in general will
not have this property. The second problem is that it is not clear
when the second condition is satisfied.
\end{rem}

\subsection{Proof of Theorem \ref{thm:when FibX,C has the resolution property}}

The entire section is dedicated to the proof of Theorem \ref{thm:when FibX,C has the resolution property}.
In particular we assume the hypothesis and notation in its statement.

It is easy to see that $\Fib_{\stX,\shC}$ is a stack (not necessarily
in groupoids) for the fpqc topology on $\Aff/R$. To avoid problems
with disjoint unions we can assume that $\stX$ is a Zariski stack.

Given a finite subset $J$ of $\shC$ denote by $\shI_{J}$ the set
of $f\colon J\to\N$ extending to a function of $\shI$. In other
words a function $f\colon J\to\N$ belongs to $\shI_{J}$ if and only
if there exists a point $s\colon\Spec L\to\stX$ such that $f(\E)=\rk s^{*}\E$
for all $\E\in J$. Given $f\colon J\to\N$ we denote by $\stX_{f}$
the open locus of $\stX$ where $\rk\E=f(\E)$ for all $\E\in J$
and set $\GL_{f}=\prod_{\E\in J}\GL_{f(\E)}$. Notice that $\stX=\bigsqcup_{f\in\shI_{J}}\stX_{f}$
and, since $\stX$ is quasi-compact, $\shI_{J}$ is finite. The sheaves
$(\E_{|\stX_{f}})_{\E\in J}$ induces a map $\stX_{f}\arr\Bi\GL_{f}=\shB_{f}$
and thus a map
\[
\omega_{J}\colon\stX\arr\bigsqcup_{f\in\shI_{J}}\shB_{f}=\shB_{J}
\]
For all $\E\in J$ there is a locally free sheaf $\shH_{\E,J}$ on
$\shB_{J}$ such that $(\shH_{\E,J})_{|\shB_{f}}$ is the canonical
locally free sheaf of rank $f(\E)$ pullback from $\Bi\GL_{f(\E)}$.
By construction $\omega_{J}^{*}\shH_{\E,J}\simeq\E$. Set $\shD_{J}$
for the subcategory of $\Loc(\shB_{J})$ consisting of all sheaves
$\shG$ such that $\omega_{J}^{*}\shG\in\shC$. This is a monoidal
subcategory with duals. Finally when $J=\{\E\}$ we will replace $J$
by $\E$ in the subscripts.

The idea now is to prove that $\shD_{J}$ generates $\QCoh(\shB_{j}),$so
that $\shP_{\shD_{J}}\colon\shB_{J}\arr\Fib_{\shB_{J},\shD_{J}}$
is an equivalence of stacks by \ref{thm: Tannaka reconstrion for Loc X}.
Using this, any time we will have a monoidal functor $\Gamma\colon\shC\to\Vect(A)$
in $\Fib_{\stX,\shC}(A)$, we can compose it via the pullback $\omega_{J}^{*}\colon\shD_{J}\to\shC$
and deduce that $\Gamma\circ\omega_{J}^{*}\simeq s^{*}$ for some
$s\colon\Spec A\to\shB_{J}$.
\begin{lem}
The category $\shD_{J}$ generates $\QCoh(\shB_{J})$ and, in particular,
$\shP_{\shB_{J}}\colon\shB_{J}\to\Fib_{\shB_{J},\shD_{J}}$ is an
equivalence.
\end{lem}

\begin{proof}
By hypothesis $\shD_{J}$ contains $\shH_{\E,J}$ and therefore $\shH_{\E,J}^{\vee}$
for $\E\in J$. Moreover if $R$ is not a $\Q$-algebra then $\shD_{J}$
contains also $\Sym^{n}\shH_{\E,J}$ and therefore $(\Sym^{n}\shH_{\E,J})^{\vee}$
for $n\in\N$ and $\E\in J$. So everything follows from \ref{lem:generators for GLn}
and \ref{thm: Tannaka reconstrion for Loc X}.
\end{proof}
\begin{lem}
The stack $\Fib_{\stX,\shC}$ is fibered in groupoids.
\end{lem}

\begin{proof}
If $\Gamma,\Gamma'\in\Fib_{\stX,\shC}(A)$, $\delta\colon\Gamma\arr\Gamma'$
is a morphism and $\E\in\shC$ then $\delta_{\E}\colon\Gamma_{\E}\arr\Gamma'_{\E}$
is an isomorphism because $\Gamma\circ\omega_{\E}^{*}\arrdi{\delta\circ\omega_{\E}^{*}}\Gamma'\circ\omega_{\E}^{*}$
is a morphism in the groupoid $\Fib_{\stB_{\E},\shD_{\E}}(A)\simeq\shB_{\E}(A)$.
\end{proof}
\begin{lem}
\label{lem:shG and operations} Let $t$ be any of the following operations
of sheaves $t=-^{\vee}$, $\Sym^{m}$, $\Lambda^{m}$ and $\E\in\shC$
be such that $t(\E)\in\shC$. Then $\shG_{t(\E)}\simeq t(\shG_{\E})$
on $\Fib_{\stX,\shC}$. In particular for any $\Gamma\in\Fib_{\stX,\shC}(A)$
we get an isomorphism $\Gamma_{t(\E)}\simeq t(\Gamma_{\E})$ by testing
the previous isomorphism in $\Gamma$.
\end{lem}

\begin{proof}
Consider $J=\{\E\}\subseteq\shC$. The functor $\omega_{\E}\colon\stX\arr\shB_{\E}$
induces $\omega_{\E}^{*}\colon\shD_{\E}\arr\shC$ and a functor $\delta\colon\Fib_{\stX,\shC}\arr\Fib_{\shB_{\E},\shD_{\E}}\simeq\shB_{\E}$.
In other words there is a canonical isomorphism 
\[
\delta(\Gamma)_{|\shD_{\E}}^{*}\simeq\Gamma\circ\omega_{\E}^{*}\colon\shD_{\E}\to\Loc A\text{ for }\Gamma\in\Fib_{\stX,\shC}(A)
\]
Since $t(\E)\simeq\omega_{\E}^{*}t(\shH_{\E,\{\E\}})\in\shC$ we have
$t(\shH_{\E,\{\E\}})\in\shD_{\E}$. So we have isomorphisms
\[
\shG_{t(\E)}(\Gamma)=\Gamma_{t(\E)}\simeq(\Gamma\circ\omega_{\E}^{*})_{t(\shH_{\E,\{\E\}})}\simeq\delta(\Gamma)^{*}t(\shH_{\E,\{\E\}})\simeq t(\delta(\Gamma)^{*}\shH_{\E,\{\E\}})\simeq t(\Gamma_{\E})\simeq t(\shG_{\E})(\Gamma)
\]
natural in $\Gamma\in\Fib_{\stX,\shC}$ and thus that $\shG_{t(\E)}\simeq t(\shG_{\E})$.
\end{proof}
\begin{prop}
\label{prop:charcaterization of shI} We have $\Fib_{\stX,\shC}^{f}\neq\emptyset$
if and only if $f\in\shI$.
\end{prop}

\begin{proof}
For the if part, if $f\in\shI$ there exists $s\colon\Spec L\arr\stX$
such that $\rk s^{*}\E=f(\E)$ for all $\E\in\shC$. Thus $s_{|\shC}^{*}\in\Fib_{\stX,\shC}^{f}(L)$.
Conversely we must show that, if $L$ is an algebraically closed field
and $\Gamma\in\Fib_{\stX,\shC}(L)$, then $\rk\Gamma_{*}\colon\shC\arr\N$
belongs to $\shI$.

Given a finite subset $J\subseteq\shC$ consider $[\Gamma\circ\omega_{J}^{*}\colon\shD_{J}\to\Loc(L)]\in\Fib_{\shB_{J},\shD_{J}}(L)\simeq\shB_{J}(L)$,
so that there exists a map $\Spec L\arrdi s\shB_{J}$ such that $\Gamma\circ\omega_{J}^{*}\simeq s_{|\shD_{J}}^{*}$.
The map $s$ has image in some component $\shB_{f}$ with $f\in\shI_{J}$.
In particular if $\E\in J$ we have 
\[
\rk\Gamma_{\E}=\rk\Gamma_{\omega_{J}^{*}\shH_{\E,J}}=\rk(s^{*}(\shH_{\E,J})_{|\shB_{f}})=f(\E)
\]
This shows that for all finite subsets $J\subseteq\shC$ we have that
$f_{J}=(\rk\Gamma_{*})_{|J}$ belongs to $\shI_{J}$. By construction
$\stX_{f_{J}}$ is a non-empty (open and) closed substack of $\stX$.
Since $\stX_{f_{J}}\subseteq\stX_{f_{J'}}$ if $J'\subseteq J$ and
$\stX$ is quasi-compact it follows that $\bigcap_{J}\stX_{f_{J}}\neq\emptyset$
and thus that $\rk\Gamma_{*}\in\shI$.
\end{proof}
\begin{lem}
With notations from \ref{subsec:Generalized-frame-bundles} we have
that $D(\E)=D(\shG_{\E})$ and $Q(\E)=Q(\shG_{\E})$. In particular
$D(\shG_{*}\colon\shC\to\Vect(\Fib_{\stX,\shC}))=D(\shC\to\Vect(\stX))$.
\end{lem}

\begin{proof}
We have to show that $\text{ranks}(\E)=\text{ranks}(\shG_{\E})$.
Since $\shP_{\shC}^{*}\shG_{\E}\simeq\shE$ we have $\text{ranks}(\E)\subseteq\text{ranks}(\shG_{\E})$.
For the converse, let $n\in\text{ranks}(\shG_{\E})$ and $\xi\colon\Spec L\to\Fib_{\stX,\shC}$
be a point such that $\xi^{*}\shG_{\E}$ has rank $n$. Let $\Gamma\in\Fib_{\stX,\shC}(L)$
be the monoidal functor corresponding to $\xi$, so that $\xi^{*}\shG_{\E}=\Gamma_{\E}$,
and set $f=\rk\Gamma\colon\shC\to\N$, so that $f(\E)=n$. Since $\Gamma\in\Fib_{\stX,\shC}^{f}\neq\emptyset$,
by \ref{prop:charcaterization of shI} we have $f\in\shI$. By definition
of $f$ there exists $t\colon\Spec\Omega\to\stX$ such that $f(\shG)=\rk t^{*}\shG$
for all $\shG\in\shC$. In particular $f(\E)=n=\rk t^{*}\E$. By definition
it follows that $n\in\text{ranks}(\E)$.
\end{proof}
We set $D(\shC)=D(\shG_{*}\colon\shC\to\Vect(\Fib_{\stX,\shC}))=D(\shC\to\Vect(\stX))$.
Since $\stX$ is quasi-compact, by \ref{rem:connected components vector affine}
we have that $D(\shC)=\Spec S$ is affine. By abuse of notation we
denote by $Q(\E)=Q(\shG_{\E})$ the pullback along $D(\shC)\to D(\E)=D(\shG_{\E})$
of the vector bundle denoted by the same symbol. By the above lemma
there is a map $\Fib_{\stX,\shC}\to D(\shC)$. We define the (pseudo)-functor
$\Omega\colon\Aff/S\to(\text{categories})$ by 
\[
\Omega(A)=\{R\text{-linear and strong monoidal functors }\Gamma\colon\shC\to\Vect(A)\text{ with fixed }\Gamma_{\E}\simeq Q(\E)\otimes_{S}S\}
\]

\begin{lem}
\label{lem:the space of non exact functors} The functor $\Omega$
is an affine scheme and $\Fr(\shG_{*}\colon\shC\to\Vect(\Fib_{\stX,\shC}))$
is the locus in $\Omega$ of functors $\Gamma$ which are right exact,
that is $\Gamma\in\Fib_{\stX,\shC}$.
\end{lem}

\begin{proof}
By definition $\Fr(\shG_{*}\colon\shC\to\Vect(\Fib_{\stX,\shC}))(A)$
is the groupoid of $\xi\in\Fib_{\stX,\shC}(A)$ together with isomorphisms
$\xi^{*}\shG_{\E}\simeq Q(\E)\otimes_{S}A$. Since objects $\xi$
correspond to $R$-linear and strong monoidal functors $\Gamma\colon\shC\to\Vect(A)$
which are right exact and $\xi^{*}\shG_{\E}\simeq\Gamma_{\E}$, the
second part of the statement is clear.

For the first one, notice that, since the $Q(\E)$ are vector bundles
on $D(\shC)$, the sheaves
\[
\Homsh_{D(\shC)}(Q(\E),Q(\E'))\comma\Isosh_{D(\shC)}(Q(\E)\otimes Q(\E'),Q(\E\otimes\E'))\comma\Isosh(\odi S,Q(\odi{\stX}))
\]
are all affine schemes for $\E,\E'\in\shC$. In particular 
\[
\overline{\Omega}=(\prod_{\E\to\E'}\Homsh_{D(\shC)}(Q(\E),Q(\E')))\times(\prod_{\E,\E'}\Isosh_{D(\shC)}(Q(\E)\otimes Q(\E'),Q(\E\otimes\E')))\times(\Isosh(\odi S,Q(\odi{\stX})))
\]
where the product is taken over $D(\shC)$, is a very big affine scheme
and $\Omega$ is a subfunctor of $\overline{\Omega}=\Spec T$.

Set $Q(\E)\otimes_{S}T=Q(\E)_{T}$. By construction there are maps
$Q(\E)_{T}\to Q(\E')_{T}$ for any map $\E\to\E'$ in $\shC$, isomorphisms
$Q(\E)_{T}\otimes Q(\E')_{T}\to Q(\E\otimes\E')_{T}$ for all $\E,\E'\in\shC$
and an isomorphism $\odi{\Omega}\to Q(\odi{\stX})_{T}$. We have that
$\Omega$ is the locus in $\overline{\Omega}$ where the association
$\E\mapsto Q(\E)_{T}$ defines an $R$-linear strong monoidal functor.
All these condition can be expressed by the vanishing of a map of
vector bundles, which are closed relations. For example the locus
where the association $Q(-)_{T}$ is symmetric is the locus where
the difference between $Q(\E)_{T}\otimes Q(\E')_{T}\to Q(\E\otimes\E')_{T}\to Q(\E'\otimes\E)$
and $Q(\E)_{T}\otimes Q(\E')_{T}\to Q(\E')_{T}\otimes Q(\E)_{T}\to Q(\E'\otimes\E')_{T}$
vanishes.

In conclusion $\Omega$ is a closed subscheme of $\overline{\Omega}$
and it is therefore affine.
\end{proof}
We now assume the notation from \ref{thm:when FibX,C has the resolution property},
$2)$. In particular we set $G=\prod_{\E}\GL(Q(\E))\to D(\shC)$.
\begin{lem}
\label{lem:general Fib pseudo-affine} The functor $\Fib_{\stX,\shC}\to\Bi G$
is pseudo-affine and, in particular, $\Fr(\shG_{*}\colon\shC\to\Vect(\Fib_{\stX,\shC}))$
is pseudo-affine. If moreover the hypothesis of \ref{thm:when FibX,C has the resolution property},
$5)$ are met then the previous map and space are quasi-affine.
\end{lem}

\begin{proof}
By \ref{rem:absolute relative pseudo-affine} and \ref{rem:fpqc descent of pseudo-affine}
it is enough to prove that $U=\Fr(\shG_{*}\colon\shC\to\Vect(\Fib_{\stX,\shC}))$
is pseudo-affine. We are going to use \ref{lem:the space of non exact functors}.
Set $\Omega=\Spec B$ and denote by $\Gamma\colon\shC\to\Vect(B)$
the tautological $R$-linear and strong monoidal functor. Given a
finite test sequence $\shT$ in $\shC$ the sequence of maps $\Gamma_{\shT}$
is a complex of free $B$-modules and denote by $\Omega_{\shT}$ the
locus in $\Omega$ where this complex is exact. By \ref{lem:locus of exactness}
the locus $\Omega_{\shT}$ is a quasi-compact open subset of $\Omega$.
Moreover $U$ is the intersection of all the $\Omega_{\shT}$ for
all finite test sequences $\shT$. The fact that $U$ is pseudo-affine
follows from \ref{prop:arbitrary intersection of quasi-compact open is pseudo-affine}.
In the hypothesis of \ref{thm:when FibX,C has the resolution property},
$5)$ only finitely many test sequences are needed and therefore $U$
is a finite intersection of the $\Omega_{\shT}$, hence it is quasi-affine.
\end{proof}
Set $\shD=\{\shG_{\E}\}_{\E\in\shC}\subseteq\Loc(\Fib_{\stX,\shC})$
and let $\shC'$ be the subcategory of $\Loc\Bi G$ obtained by taking
tensor products of the sheaves considered in \ref{lem:generators for GLn}
with respect to the map $\shC\to\Vect(D(\shC))$, $\E\mapsto Q(\E)$.
By \ref{lem:generators for GLn} the category $\shC'$ generates $\QCoh(\Bi G)$.
Set also $\omega\colon\Fib_{\stX,\shC}\to\Bi G$.
\begin{lem}
We have $\omega^{*}\shC'\subseteq\shD$ and, in particular, $\shD$
generates $\QCoh(\Fib_{\stX,\shC})$.
\end{lem}

\begin{proof}
The second statement follows from the first using \ref{rem: generating subcategory under pseudo-affine maps}
and the fact that $\omega$ is pseudo-affine. For the first, denote
by $\shF_{\E}$ the tautological sheaves defined over $\Bi G$ as
in \ref{lem:generators for GLn}. By construction $\omega^{*}\shF_{\E}\simeq\shG_{\E}$.
Using \ref{lem:shG and operations} we have $\shG_{\duale{\E}}\simeq\duale{(\shG_{\E})}\simeq g^{*}(\shF_{\E})^{\vee}$
and, if $\Sym^{n}\E\in\shC$, $\shG_{\Sym^{n}\E}\simeq\Sym^{n}\shG_{\E}\simeq g^{*}(\Sym^{n}\shF_{\E})$
over $\Fib_{\stX,\shC}$, which concludes the proof.
\end{proof}
We now deal with the case of $\Fib_{\stX,\shC}^{f}$ for $f\in\shI$.
The map $\Spec R\to D(\shC)$ defined in \ref{thm:when FibX,C has the resolution property},
$4)$ is a flat closed immersion thanks to \ref{lem:rank loci}. Moreover
it is easy to check that $G\times_{D(\shC)}R=\GL_{f}$ and therefore
$(\Bi_{D(\shC)}G)\times_{D(\shC)}R=\Bi\GL_{f}$. Moreover by construction
$\Fib_{\stX,\shC}^{f}$ is the locus of $\Fib_{\stX,\shC}$ where
all the $\shG_{\E}$ have rank $f(\E)$. Hence $\Fib_{\stX,\shC}^{f}=\Fib_{\stX,\shC}\times_{D(\shC)}R$.
Taking into account \ref{rem: generating subcategory under pseudo-affine maps}
we obtain:
\begin{prop}
Claims in Theorem \ref{thm:when FibX,C has the resolution property},
$4)$ hold.
\end{prop}

It remains to show Theorem \ref{thm:when FibX,C has the resolution property},
$5)$. So we assume its hypothesis. We have already shown that $U=\Fr(\shG_{*}\colon\shC\to\Vect(\Fib_{\stX,\shC}))$
is quasi-affine in \ref{lem:general Fib pseudo-affine}. For $T\subseteq\shC$
set $G_{T}=\prod_{\E\in T}\GL(Q(\E))\to D(\shC)$. In particular $G=G_{\shC}=G_{J}\times G_{\shC-J}$
and $\Bi G=\Bi G_{J}\times\Bi G_{\shC-J}$. Consider the Cartesian
diagram     \[   \begin{tikzpicture}[xscale=2.2,yscale=-1.2]     
\node (A0_0) at (0, 0) {$U$};     
\node (A0_1) at (1, 0) {$D(\shC)$};     
\node (A1_0) at (0, 1) {$U'$};     
\node (A1_1) at (1, 1) {$\Bi G_{\shC-J}$};     
\node (A1_2) at (2, 1) {$D(\shC)$};     
\node (A2_0) at (0, 2) {$\Fib_{\stX,\shC}$};     
\node (A2_1) at (1, 2) {$\Bi G$};     
\node (A2_2) at (2, 2) {$\Bi G_J$};     
\path (A2_1) edge [->]node [auto] {$\scriptstyle{}$} (A2_2);     \path (A0_0) edge [->]node [auto] {$\scriptstyle{}$} (A0_1);     \path (A2_0) edge [->]node [auto] {$\scriptstyle{}$} (A2_1);     \path (A1_0) edge [->]node [auto] {$\scriptstyle{}$} (A1_1);     \path (A1_1) edge [->]node [auto] {$\scriptstyle{}$} (A1_2);     \path (A1_0) edge [->]node [auto] {$\scriptstyle{\alpha}$} (A2_0);     \path (A1_1) edge [->]node [auto] {$\scriptstyle{}$} (A2_1);     \path (A0_0) edge [->]node [auto] {$\scriptstyle{}$} (A1_0);     \path (A0_1) edge [->]node [auto] {$\scriptstyle{}$} (A1_1);     \path (A1_2) edge [->]node [auto] {$\scriptstyle{}$} (A2_2);   \end{tikzpicture}   \] 
\begin{lem}
The map $U\to U'$ has a section and, in particular, $U'$ is a closed
subscheme of $U$, hence quasi-affine.
\end{lem}

\begin{proof}
We need to show that the torsor $U'\to\Bi G_{\shC-J}$ is trivial.
For this, it is enough to show that $U'\to\Bi G$ is trivial. This
last map is given by the collection $(\alpha^{*}\shG_{\E})_{\E\in\shC}$
and therefore we must show that all those sheaves are free.

Let $\shS\subseteq\shC$ be the subset of $\E$ such that $\alpha^{*}\shG_{\E}$
is free. By construction $J\subseteq\shS$. On the other hand in \ref{lem:shG and operations}
we have seen that if $\E\in\shC$ is such that $t(\E)\in\shC$ then
$\shG_{t(\E)}\simeq t(\shG_{\E})$ for any of the following operations
of sheaves $t=-^{\vee}$, $\Sym^{m}$, $\Lambda^{m}$. In particular
$\alpha^{*}\shG_{t(\E)}\simeq t(\alpha^{*}\shG_{\E})$ and therefore
$\E\in\shS$ implies $t(\E)\in\shS.$ By a direct check it is also
clear that $\shG_{\E\otimes\E'}\simeq\shG_{\E}\otimes\shG_{\E'}$
for $\E,\E'\in\shC$ and, if $\E\oplus\E'\in\shC$, that $\shG_{\E\oplus\shE'}\simeq\shG_{\E}\oplus\shG_{\E'}$.
Thus $\shS$ is closed under tensor product and, when it makes sense,
under direct sum.

Using the fist condition on $\shC$ we can conclude that $\shS=\shC$.
Since $U\to U'$ is a $G_{\shC-J}$-torsor and therefore affine, the
existence of a section tell us that $U'$ is a closed subscheme of
$U$ and therefore is quasi-affine.
\end{proof}
As a consequence of the previous result we have that $\Fib_{\stX,\shC}\simeq[U'/G_{J}]$
with $U'$ quasi-affine. We want to show that $\Fib_{\stX,\shC}\simeq[V/\GL_{n}]$
for some quasi-affine scheme $V$. Since $J$ is finite, there is
a finite decomposition of $D(\shC)=\sqcup_{i}D_{i}$ into open and
closed subsets such that $Q(\E)_{|U_{i}}$ has constant rank for $\E\in J$.
In particular we have that $(G_{J})_{|U_{i}}$ is a finite product
of $\GL_{m}$'s. The following lemma concludes the proof of Theorem
\ref{thm:when FibX,C has the resolution property}, $5)$ and therefore
of Theorem \ref{thm:when FibX,C has the resolution property} itself.
\begin{lem}
Let $m_{1},\dots,m_{h}\in\N$ and $n>m_{1}+\cdots+m_{h}$. Then there
is an injective map
\[
G=\GL_{m_{1}}\times\cdots\times\GL_{m_{h}}\to\GL_{n}
\]
over $\Z$ such that $\Bi G\to\Bi\GL_{n}$ is quasi-affine. In particular
for any quasi-affine scheme $U$ with an action of $G$ one can write
\[
[U/G]\simeq[V/\GL_{n}]
\]
for some quasi-affine scheme $V$ with an action of $\GL_{n}$.
\end{lem}

\begin{proof}
The last claim follows from the first because $[U/G]\to\Bi G\to\Bi\GL_{n}$
would be quasi-affine because composition of quasi-affine maps.

Consider as $G\to\GL_{n}$ the map sending matrices $M_{1},\dots,M_{h}$
to the $n\times n$ matrix having the $M_{j}$ on the diagonal and
$1$ in the remaining spots of the diagonal. The corresponding map
$\Bi G\to\Bi\GL_{n}$, in terms of vector bundles, is induced by $\shH=(\bigoplus_{i}\shF_{i})\oplus\odi{}^{n-m}$,
where $m=\sum_{i}m_{i}$ and $\shF_{i}$ is the rank $m_{i}$ tautological
vector bundle given by $\Bi G\to\Bi G_{i}$. Denote by $h\colon W\to\Bi G$
the corresponding $\GL_{n}$-torsor. We have to show that $W$ is
quasi-affine and, since $W$ is an algebraic space, it is enough to
show that it is pseudo-affine thanks to \ref{cor:pseudo-affine algebraic is quasi-affine}.
It is enough to show that the sheaves over $\Bi G$ coming from $\Bi\GL_{n}$,
whose category we denote by $\shD_{n}$, generates $\QCoh(\Bi G)$:
$h^{*}\shD_{n}$ are made of free sheaves and, since $h$ is affine,
it generates $\QCoh(W)$ thanks to \ref{rem: generating subcategory under pseudo-affine maps},.

We have $\shH=\shF_{i}\oplus Q_{i}$. In particular there is a surjective
map $\shH\to\shF_{i}$. Moreover, since $\Sym^{q}(\shH)=(\Sym^{q}\shF_{\E})\oplus\shK_{i}$,
there is also a surjective map $\shD_{n}\ni(\Sym^{m}\shH)^{\vee}\to(\Sym^{m}\shF_{i})^{\vee}$.
It follows that the sheaves $\shF_{i}$, $(\Sym^{q}\shF_{i})^{\vee}$
for $q\in\N$ are generated by $\shD_{n}$, which therefore generates
$\QCoh(\Bi G)$ thanks to \ref{lem:generators for GLn}.
\end{proof}

\subsection{Comparing vector bundles}

In this section we want to discuss some relations between the vector
bundles on $\stX$ and on $\Fib_{\stX,\shC}$.
\begin{prop}
\label{prop:universality equivalent conditions} Let $\stX$ and $\shC$
be as in \ref{thm:when FibX,C has the resolution property}. The following
conditions are equivalent:
\begin{enumerate}
\item (if $\stX$ is quasi-separated (see \cite[Def 1.4]{Tonini2014}))
the morphism $\odi{\Fib_{\stX,\shC}}\to\shP_{\shC*}\odi{\stX}$ is
an isomorphism;
\item the functor $\Vect(\Fib_{\stX,\shC})\to\Vect(\stX)$ is fully faithful;
\item the functor $\shG_{*}\colon\shC\to\Vect(\Fib_{\stX,\shC})$ is fully
faithful.
\item the functor $\shG_{*}(\shC)\to\Vect(\stX)$ is faithful
\end{enumerate}
\end{prop}

\begin{proof}
$1.\iff2.$ We assume $\stX$ quasi-compact and quasi-separated. Since
$\Fib_{\stX,\shC}$ is quasi-compact and quasi-separated by \ref{thm:when FibX,C has the resolution property},
the functor $\shP_{\shC}\colon\stX\to\Fib_{\stX,\shC}$ is quasi-compact
and quasi-separated as well, so that pushing forward quasi-coherent
sheaves works as expected (see \cite[Prop 1.15]{Tonini2014}). In
particular $\shP_{\shC*}\odi{\stX}$ is a quasi-coherent sheaf.

For $\shH\in\QCoh(\Fib_{\stX,\shC})$ consider the morphism
\[
\delta_{\shH}\colon\Hom(\shH,\odi{\Fib_{\stX,\shC}})\to\Hom(\shH_{|\stX},\odi{\stX})
\]
This map can also be obtained applying $\Hom(\shH,-)$ on the map
$\odi{\Fib_{\stX,\shC}}\to\shP_{\shC*}\odi{\stX}$:
\[
\Hom(\shH,\odi{\Fib_{\stX,\shC}})\to\Hom(\shH,\shP_{\shC*}\odi{\stX})\simeq\Hom(\shH_{|\stX},\odi{\stX})
\]
In particular Yoneda's lemma tells us that $1.$ is equivalent to
the fact that $\delta_{\shH}$ is an isomorphism for all $\shH\in\QCoh(\Fib_{\stX,\shC})$.
On the other hand, condition $2.$ is equivalent to the fact that
all maps $\delta_{\shH}$ are isomorphisms for $\shH\in\Vect(\Fib_{\stX,\shC})$.
Thus clearly $1.\then2.$. The converse instead follows from \ref{thm:main thm main paper}:
$\QCoh(\Fib_{\stX,\shC})\to\L_{R}(\Vect(\Fib_{\stX,\shC}),R)$ is
fully faithful because $\Fib_{\stX,\shC}$ has the resolution property
by \ref{thm:when FibX,C has the resolution property}.

$2.\then3.\iff4$. It is enough to recall that $\shC\xrightarrow{\shG_{*}}\Vect(\Fib_{\stX,\shC})\to\Vect(\stX)$
is the inclusion.

$4.\then2.$ We have to show that the map
\[
\Hom(\shH,\shH')\to\Hom(\shH_{|\stX},\shH'_{|\stX'})
\]
is an isomorphism for all $\shH,\shH'\in\Vect(\Fib_{\stX,\shC})$.
Since $\shG_{*}(\shC)$ generates $\QCoh(\stX)$ by \ref{thm:when FibX,C has the resolution property}
and $\shH$ is finitely presented, there exists an exact sequence
\[
\shG_{\E'}\to\shG_{\E}\to\shH\to0
\]
for some $\E,\E'\in\shC^{\oplus}$. As $\Hom(-,\shH')$ and $\Hom(-_{|\stX},\shH'_{|\stX}$)
are both exact on the above sequence, we can assume $\shH=\shG_{\overline{\E}}$
for some $\overline{\E}\in\shC^{\oplus}$.

Considering a resolution for $\shH'^{\vee}$, dualizing and using
that $\shG_{\E}^{\vee}\simeq\shG_{\E^{\vee}}$ by \ref{lem:shG and operations},
we also obtain an exact sequence 
\[
0\to\shH\to\shG_{\E}\to\shG_{\E'}
\]
for some $\E,\E'\in\shC^{\oplus}$. Moreover also this sequence restricts
to an exact sequence on $\stX$ because $\shH'$ is locally free:
$\Coker(\shH\to\shG_{\E})$ is locally free. In particular $\Hom(\shG_{\overline{\E}},-)$
and $\Hom(\overline{\E},-_{|\stX})$ are exact on the above sequence
and therefore we can also assume $\shH'=\shG_{\E}$ for some $\E\in\shC^{\oplus}$.

Splitting $\overline{\E}$ and $\E$ into direct sums of sheaves in
$\shC$ we are left to prove that
\[
\Hom(\shG_{\overline{\E}},\shG_{\E})\to\Hom(\overline{\E},\E)
\]
is an isomorphism for $\overline{\E},\E\in\shC$. By hypothesis this
map in injective. On the other hand it is surjective because $\shC\xrightarrow{\shG_{*}}\Vect(\Fib_{\stX,\shC})\to\Vect(\stX)$
is the inclusion.
\end{proof}
\begin{conjecture}
\label{conj:universal general} Let $\stX$ and $\shC$ be as in \ref{thm:when FibX,C has the resolution property}.
Then the equivalent conditions of Proposition \ref{prop:universality equivalent conditions}
holds.
\end{conjecture}

\begin{rem}
Conjecture \ref{conj:universal general} holds in two special cases.
The first one is when $\stX$ is quasi-compact with quasi-affine diagonal
and $\shC$ generates $\QCoh(\stX)$: by \ref{thm: Tannaka reconstrion for Loc X}
the map $\stX\to\Fib_{\stX,\shC}$ is actually an equivalence. The
second is when $\stX$ is quasi-compact and quasi-separated and $\shC=\{\odi{\stX}\}$,
the simplest category: the Conjecture is verified thanks to \ref{prop:the simplest category},
that we are going to prove.
\end{rem}

\begin{prop}
Let $\stX$ be a quasi-compact fibered category over $R$ and $\stY$
a quasi-compact fpqc stack with quasi-affine diagonal and with the
resolution property. Then any map $f\colon\stX\to\stY$ has a canonical
factorization $\stX\to\Fib_{\stX,\Vect(\stX)}\to\stY$.
\end{prop}

\begin{proof}
The canonical extension $\Fib_{\stX,\Vect(\stX)}\to\stY$ is the composition
of $\Fib_{\stX,\Vect(\stX)}\to\Fib_{\stY,\Vect(\stY)}$, $\Gamma\mapsto\Gamma\circ f^{*}$
and the equivalence $\stY\to\Fib_{\stY,\Vect(\stY)}$ (see \ref{thm: Tannaka reconstrion for Loc X}).
This can also be described, again using \ref{thm: Tannaka reconstrion for Loc X},
as the map corresponding to the composition
\[
\Vect(\stY)\to\Vect(\stX)\xrightarrow{\shG_{*}}\Vect(\Fib_{\stX,\Vect(\stX)})
\]
\end{proof}
\begin{conjecture}
\label{conj:universal resolution property} Let $\stX$ be a quasi-compact
fibered category over $R$. The morphism $\stX\to\Fib_{\stX,\shC}$
is universal among maps from $\stX$ to quasi-compact fpqc stacks
with quasi-affine diagonal and with the resolution property, that
is 
\begin{equation}
\Hom(\Fib_{\stX,\Vect(\stX)},\stY)\to\Hom(\stX,\stY)\label{eq:universal}
\end{equation}
is an equivalence for all such stacks $\stY$.
\end{conjecture}

Proposition below shows that Conjecture \ref{conj:universal general}
for $\shC=\Vect(\stX)$ is equivalent to Conjecture \ref{conj:universal resolution property}.
\begin{prop}
\label{prop:Conjecture implications} Conjecture \ref{conj:universal resolution property}
holds if and only if
\[
\Vect(\Fib_{\stX,\Vect(\stX)})\to\Vect(\stX)
\]
 is faithful. In this case the above map is an equivalence.
\end{prop}

\begin{proof}
As the composition $\Vect(\stX)\xrightarrow{\shG_{*}}\Vect(\Fib_{\stX,\Vect(\stX)})\to\Vect(\stX)$
is the identity, we see that if $\Vect(\Fib_{\stX,\Vect(\stX)})\to\Vect(\stX)$
is fully faithful then it is an equivalence. Moreover by \ref{prop:universality equivalent conditions}
faithfulness also implies the fullness. Assume this is true. If we
start by a functor $\Fib_{\stX,\Vect(\stX)}\to\stY$, which is completely
determined by the pullback
\[
\Vect(\stY)\to\Vect(\Fib_{\stX,\Vect(\stX)})
\]
by \ref{thm: Tannaka reconstrion for Loc X}, the new map $\Fib_{\stX,\Vect(\stX)}\to\stY$
obtained corresponds to
\[
\Vect(\stY)\to\Vect(\Fib_{\stX,\Vect(\stX)})\to\Vect(\stX)\xrightarrow{\shG_{*}}\Vect(\Fib_{\stX,\Vect(\stX)})
\]
and therefore is canonically isomorphic to the original one.

Conversely assume $\stX\to\Fib_{\stX,\Vect(\stX)}$ is universal.
Given $h\in\N$ we define $\stY_{h}$ as the stack of vector bundles
whose local rank is bounded above by $h$. In other words 
\[
\stY_{h}=\bigsqcup_{l=0}^{h}\Bi\Gl_{l}
\]
which is a quasi-compact stack with affine diagonal and the resolution
property. Assuming that (\ref{eq:universal}) is an equivalence for
all $\stY_{h}$ means that 
\[
\Vect(\Fib_{\stX,\Vect(\stX)})\to\Vect(\stX)
\]
induces an equivalence on the corresponding groupoids. We have to
show that, if $\E,\shH\in\Vect(\Fib_{\stX,\Vect(\stX)})$ and $\phi\colon\E\to\shH$
is a morphism and $\phi_{|\stX}=0$ then $\phi=0$. Consider the isomorphism
\[
\psi\colon(\E\oplus\shH)\to(\E\oplus\shH)\comma e\oplus h\longmapsto e\oplus(h+\phi(e))
\]
By construction $\psi_{|\stX}=\id$. Thus $\psi=\id$ and $\phi=0$
as required. 
\end{proof}

\subsection{The case $\protect\shC=\{\protect\odi{\protect\stX}\}$}

In this section we consider the simple case $\shC=\{\odi{\stX}\}$.
Notice that $\Sym^{n}\odi{\stX}\simeq\odi{\stX}$ so the hypothesis
of \ref{thm:when FibX,C has the resolution property} are satisfied.
\begin{prop}
\label{prop:the simplest category}Let $\stX$ be a quasi-compact
fibered category over $R$. Then $\Fib_{\stX,\{\odi{\stX}\}}$ is
the intersection of all quasi-compact open subsets of $\Spec\Hl^{0}(\odi{\stX})$
containing the image of $\stX\to\Spec\Hl^{0}(\odi{\stX})$.

If $\stX$ is quasi-compact and quasi-separated then $\Vect(\Fib_{\stX,\{\odi{\stX}\}})\to\Vect(\stX)$
is fully faithful, or, equivalently, $\odi{\Fib_{\stX,\{\odi{\stX}\}}}\to\shP_{\{\odi{\stX}\}*}\odi{\stX}$
is an isomorphism. Moreover $\stX\to\Fib_{\stX,\{\odi{\stX}\}}$ is
universal among maps from $\stX$ to pseudo-affine sheaves.
\end{prop}

\begin{proof}
We can assume $R=\Hl^{0}(\odi{\stX})$. Set $\shC=\{\odi{\stX}\}$.
If $A$ is an $R$-algebra, there is a unique $R$-linear and strong
monoidal functor $\Gamma^{A}\colon\shC\to\Vect(A)$, namely $\Gamma^{A}=\Hl^{0}(-)\otimes_{R}A$.
This means that $\Fib_{\stX,\shC}$ is a subfunctor of $\Spec R$.
We have that $\Gamma^{A}\in\Fib_{\stX,\shC}(A)$ if and only if, given
a test sequence $\shT_{*}$ for $\shC$, $\Gamma^{A}$ is exact on
it. As
\[
\Gamma^{A}(\shT_{*})=\Hl^{0}(\shT_{*})\otimes_{R}A
\]
we can conclude that $\Fib_{\stX,\shC}$ is the locus where all complex
of free $R$-modules $\Hl^{0}(\shT_{*})$ are exact. By \ref{lem:locus of exactness},
this is an intersection of quasi-compact open subsets of $\Spec R$
containing the image of $\stX\to\Spec R$. Conversely if $\stX\to U\subseteq\Spec R$
and $U$ is a quasi-compact open subset, say $U=\Spec R-V(a_{1},\dots,a_{n})$,
then $U$ is the locus where $(a_{1},\dots,a_{n})\colon R^{n}\to R$
is surjective and, in particular, 
\[
\odi{\stX}^{n}\xrightarrow{(a_{1},\dots,a_{n})}\odi{\stX}\to0
\]
is a test sequence. Therefore $\Fib_{\stX,\shC}\subseteq U$ as required.

Now assume that $\stX$ (and therefore $\shP_{\shC}\colon\stX\to\Fib_{\stX,\shC}$)
is quasi-compact and quasi-separated. Set $\pi\colon\stX\to\Spec R$.
By construction $\pi_{*}\odi{\stX}=\odi{\Spec R}$. From the Cartesian
diagram   \[   \begin{tikzpicture}[xscale=2.0,yscale=-1.2]     \node (A0_0) at (0, 0) {$\stX$};     \node (A0_1) at (1, 0) {$\stX$};     \node (A1_0) at (0, 1) {$\Fib_{\stX,\shC}$};     \node (A1_1) at (1, 1) {$\Spec R$};     \path (A0_0) edge [->]node [auto] {$\scriptstyle{\id}$} (A0_1);     \path (A0_0) edge [->]node [auto] {$\scriptstyle{\shP_\shC}$} (A1_0);     \path (A0_1) edge [->]node [auto] {$\scriptstyle{\pi}$} (A1_1);     \path (A1_0) edge [->]node [auto] {$\scriptstyle{u}$} (A1_1);   \end{tikzpicture}   \] the
flatness of $u$ and \cite[Prop 1.15]{Tonini2014} we can conclude
that 
\[
\odi{\Fib_{\stX,\shC}}\simeq u^{*}\odi{\Spec R}\simeq u^{*}(\pi_{*}\odi{\stX})\simeq\shP_{\shC*}\odi{\stX}
\]
By \ref{prop:universality equivalent conditions} this is equivalent
to the fully faithfulness of $\Vect(\Fib_{\stX,\shC})\to\Vect(\stX)$.

For the last claim, if $\stX\to U$ is a map to a quasi-affine scheme
consider the diagram    \[   \begin{tikzpicture}[xscale=2.0,yscale=-1.2]     \node (A0_0) at (0, 0) {$\stX$};     \node (A0_1) at (1, 0) {$\Fib_{\stX,\shC}$};     \node (A0_2) at (2, 0) {$V$};     \node (A0_3) at (3, 0) {$\Spec R$};     \node (A1_2) at (2, 1) {$U$};     \node (A1_3) at (3, 1) {$\Spec \Hl^0(\odi U)$};     \path (A0_0) edge [->,bend left=15]node [auto] {$\scriptstyle{}$} (A1_2);     \path (A0_1) edge [->,dashed]node [auto] {$\scriptstyle{}$} (A0_2);     \path (A0_3) edge [->]node [auto] {$\scriptstyle{}$} (A1_3);     \path (A0_2) edge [->]node [auto] {$\scriptstyle{}$} (A1_2);     \path (A0_0) edge [->]node [auto] {$\scriptstyle{}$} (A0_1);     \path (A0_0) edge [->,bend right=40]node [auto] {$\scriptstyle{}$} (A0_2);     \path (A1_2) edge [->]node [auto] {$\scriptstyle{}$} (A1_3);     \path (A0_2) edge [->]node [auto] {$\scriptstyle{}$} (A0_3);   \end{tikzpicture}   \] where
the square one is Cartesian. The dashed arrow exists thanks to the
description of $\Fib_{\stX,\shC}$ as intersection. From the same
diagram we can see that $\Fib_{\stX,\shC}\to U$ is uniquely determined.
The general case of a pseudo-affine sheaf follows easily from \ref{prop:arbitrary intersection of quasi-compact open is pseudo-affine}.
\end{proof}
\begin{prop}
\label{prop:baby case algebraic}Let $\stX$ be a quasi-compact fibered
category. If $\stX\to\Spec\Hl^{0}(\odi{\stX})$ has open image then
$\Fib_{\stX,\{\odi{\stX}\}}$ coincides with this image. Conversely
if $\Fib_{\stX,\{\odi{\stX}\}}$ is an algebraic stack, $\stX\to\Spec R$
is a finitely presented algebraic stack and $\Hl^{0}(\odi{\stX})$
is a Noetherian Jacobson ring then $\stX\to\Spec\Hl^{0}(\odi{\stX})$
has open image.
\end{prop}

\begin{proof}
Set $\shC=\{\odi{\stX}\}$. The first claim follows from \ref{prop:the simplest category},
so let us focus on the converse.

We can assume $R=\Hl^{0}(\odi{\stX})$. By \ref{cor:pseudo-affine algebraic is quasi-affine}
and \ref{prop:the simplest category}, the fact that $\Fib_{\stX,\shC}$
is algebraic means that it is quasi-affine. Moreover, since $\Hl^{0}(\odi{\Fib_{\stX,\shC}})=R$,
it follows that $V=\Fib_{\stX,\shC}$ is open inside $\Spec R$. By
Chevalley's theorem \cite[Tag 054K]{SP014} the image $U$ of $\stX\to\Spec R$
is constructible. In conclusion, from \ref{prop:the simplest category},
we can conclude that $V$ is the smallest open of $\Spec R$ containing
$U$. Assume by contradiction $U\neq V$. As $V-U$ is constructible,
we need to show that a constructible subset $Q$ of $\Spec R$ contains
a closed subset. We can easily reduce to the case that $Q$ is open.
As $R$ is a Jacobson ring, $Q$ contains a closed point of $\Spec R$
as required.
\end{proof}
\begin{example}
\label{exa:counterexample} We show an example of an integral scheme
$X$ of finite type over a field $k$ such that
\begin{itemize}
\item the stack $\Fib_{X,\{\odi X\}}$ is not algebraic.
\item the map $X\to\Fib_{X,\{\odi X\}}$ is not surjective, which shows
that not all fiber functors (locally) comes from a section of $X$.
\end{itemize}
By \ref{prop:baby case algebraic}, the idea is to look for an $X$
such that the image of $X\to\Spec\Hl^{0}(\odi{\stX})$ is not open.
Let $A=k[x,y,z]$ and $B=k[x/y,y,z/y]$ and notice that
\[
A\subseteq B\subseteq A_{y}\subseteq A_{xy}=B_{x}
\]
We define $X$ as the scheme obtained by gluing $\Spec A_{x}$ and
$\Spec B$ along $\Spec A_{xy}$. By construction there are maps 
\[
\Spec A-V(x,y)=\Spec A_{x}\cup\Spec A_{y}\to X\xrightarrow{f}\Spec A
\]
As the first map is dominant because $\Spec A_{x}\subseteq X$ is
open, we obtain that $\Hl^{0}(\odi{\stX})=A$. Moreover 
\[
f^{-1}(V(x,y))=\Spec(B/(x,y))=\Spec(k[x/y,z/y])\to\Spec A
\]
is given by $A\to k[x/y,z/y]$, $x,y,z\to0$. It follows that the
image of $X\to\Spec A$ is $\Imm=\Spec A-V(x,y)\sqcup\{(x,y,z)\}$.

By \ref{prop:baby case algebraic} the stack $\Fib_{X,\{\odi X\}}$
cannot be algebraic because $\Imm$ is not open. Moreover we claim
that $P=(x,y)\in\Fib_{X,\{\odi X\}}-\Imm$. By \ref{prop:the simplest category}
we have to show that if $\Imm\subseteq U\subseteq\Spec A$ is open
then $P\in U$. This is true because otherwise we would have the contradiction
\[
P\in(\Spec A-U)\subseteq\Spec A-\Imm=V(P)-\{(x,y,z)\}\subseteq V(P)=\overline{\{P\}}
\]
\end{example}

\bibliographystyle{amsalpha}
\bibliography{biblio2}

\end{document}